\documentclass{amsart}
\usepackage{amssymb}
\usepackage{mathrsfs}
\usepackage{cases}
\usepackage{amsmath}
\usepackage{amsfonts}
\usepackage{pifont}
\usepackage{graphicx,amssymb,mathrsfs,amsmath}
\usepackage{tocvsec2}
\newtheorem{thm}{Theorem}[section]

\newtheorem{prop}[thm]{Proposition}
\theoremstyle{definition}
\newtheorem{defn}[thm]{Definition}
\newtheorem{example}[thm]{Example}
\theoremstyle{remark}
\newtheorem{rem}[thm]{Remark}
\numberwithin{equation}{section}
\begin{document}
\title[Metrical approximations of functions]{Metrical approximations of functions}

\author{Belkacem Chaouchi}
\address{Lab. de l'Energie et des Syst\`{e}mes Intelligents,
Khemis Miliana University, 44225, Algeria}
\email{chaouchicukm@gmail.com}

\author{Marko Kosti\' c}
\address{Faculty of Technical Sciences,
University of Novi Sad,
Trg D. Obradovi\' ca 6, 21125 Novi Sad, Serbia}
\email{marco.s@verat.net}

\author{Daniel Velinov}
\address{Department for Mathematics, Faculty of Civil Engineering, Ss. Cyril and Methodius University, Skopje,
Partizanski Odredi
24, P.O. box 560, 1000 Skopje, N. Macedonia}
\email{velinovd@gf.ukim.edu.mk}

{\renewcommand{\thefootnote}{} \footnote{
\\ \text{  }  \ \    {\it Key words and phrases.} 
Metrical approximations of functions by trigonometric polynomials, metrical approximations of functions by $\rho$-periodic type functions, abstract Volterra integro-differential equations.
\\  \text{  }  
This research is partially supported by grant 174024 of
Ministry of Science and Technological Development, Republic of Serbia and
Bilateral project between MANU and SANU.}}

\begin{abstract}
In this paper, we analyze metrical approximations of functions $F :\Lambda \times X \rightarrow Y$ by trigonometric polynomials and $\rho$-periodic type functions, where $\emptyset \neq \Lambda \subseteq {\mathbb R}^{n}$, $X$ and $Y$ are complex Banach spaces, and $\rho$ is a general binary relation on $Y$. 
Besides the classical concept, we analyze Stepanov, Weyl, Besicovitch and Doss generalized approaches to metrical approximations.
We clarify many structural properties of introduced spaces of functions and provide several applications of our theoretical results to
the abstract Volterra integro-differential equations and the partial differential equations.
\end{abstract}
\maketitle

\section{Introduction and preliminaries}

The notion of almost periodicity was introduced by the Danish mathematician H. Bohr around 1924-1926 and later generalized by many others (see the research monographs \cite{besik}, \cite{diagana}, \cite{fink}, \cite{gaston}, \cite{nova-mono}, \cite{nova-selected}, \cite{188}, \cite{levitan}, \cite{pankov} and \cite{30} for further information concerning almost periodic functions and their applications).
Suppose that $(X,\| \cdot \|)$ is a complex Banach space, and $F : {\mathbb R}^{n} \rightarrow X$ is a continuous function, where $n\in {\mathbb N}$. Then we say that $F(\cdot)$ is almost periodic if and only if for each $\epsilon>0$
there exists a finite real number $l>0$ such that for each ${\bf t}_{0} \in {\mathbb R}^{n}$ there exists ${\bf \tau} \in B({\bf t}_{0},l)\equiv \{ {\bf t} \in {\mathbb R}^{n} : |{\bf t}-{\bf t}_{0}|\leq l\}$ such that
\begin{align*}
\bigl\|F({\bf t}+{\bf \tau})-F({\bf t})\bigr\| \leq \epsilon,\quad {\bf t}\in {\mathbb R}^{n};
\end{align*}
here, $|\cdot -\cdot|$ denotes the Euclidean distance in ${\mathbb R}^{n}.$
Equivalently, $F(\cdot)$ is almost periodic if and only if for any sequence $({\bf b}_k)$ in ${\mathbb R}^{n}$ there exists a subsequence $({\bf a}_{k})$ of $({\bf b}_k)$
such that the sequence of translations $(F(\cdot+{\bf a}_{k}))$ converges in $C_{b}({\mathbb R}^{n}: X),$ the Banach space of all bounded continuous functions on ${\mathbb R}^{n},$ equipped with the sup-norm. Let us recall that any trigonometric polynomial in ${\mathbb R}^{n}$ is almost periodic, as well as that a continuous function $F(\cdot)$ is almost periodic if and only if there exists a sequence of trigonometric polynomials in ${\mathbb R}^{n}$ which converges uniformly to $F(\cdot)$; see also \cite{marko-manuel-ap,rho,ejmaa-2022} for some recent results about the multi-dimensional almost periodic type functions.

The first systematic study of metrical almost periodicity was conducted by the second named author in 2021 (\cite{metrical}). The Stepanov, Weyl and Besicovitch classes of metrical almost periodic functions were considered in \cite{metrical-stepanov}, \cite{metrical-weyl} and \cite{multi-besik}, respectively. As already mentioned in the abstract, the main aim of this research study is to consider the metrical approximations of functions $F :\Lambda \times X \rightarrow Y$ by trigonometric polynomials and $\rho$-periodic type functions, where $\emptyset \neq \Lambda \subseteq {\mathbb R}^{n}$, $X$ and $Y$ are complex Banach spaces and $\rho$ is a general binary relation on $Y$ (it would be very difficult to summarize here all relevant results concerning function spaces obtained by the closures of the set of trigonometric polynomials in certain norms; see, e.g., the investigations of A. Oliaro, L. Rodino, P. Wahlberg \cite {oliaro} and M. A. Shubin \cite{shubin32} for some non-trivial results established in this direction). 
We also investigate the generalized Stepanov, Weyl, Besicovitch and Doss approaches to the metrical approximations of functions, and provide certain applications to
the abstract Volterra integro-differential equations. Besides many other novelties of this work, we would like to emphasize here that this is probably the first research article which examines the notion of metrical semi-periodicity, even in the one-dimensional framework.

The organization and main ideas of this research article can be briefly described as follows. After collecting some preliminary material, we introduce the notion of strong $(\phi,{\mathbb F},{\mathcal B},{\mathcal P})$-almost periodicity (semi-$(\phi,\rho,{\mathbb F},{\mathcal B},{\mathcal P})$-periodicity, semi-$(\phi,\rho_{j},{\mathbb F},{\mathcal B},{\mathcal P})_{j\in {\mathbb N}_{n}}$-periodicity)  in Definition \ref{strong-app}. The main purpose of Proposition \ref{kimih} is to consider the compositions of functions introduced here with the Lipschitz type functions. After that, in Example \ref{ferplay}, we provide many engaging examples justifying the introduction of  function spaces under our consideration (although the main aim of this paper is to create the abstract theory of metrical approximations of functions by trigonometric polynomials and $\rho$-periodic type functions, there are many places where we consider some special pseudometric spaces and specify our general notion; see also Example \ref{kakad321} and Example \ref{kakad} below).

Subsection \ref{zeljo} is devoted to the study of metrical normality and metrical Bohr type definitions. The notion of $(\phi,{\mathrm R}, {\mathcal B},{\mathbb F},{\mathcal P})$-normality is introduced in Definition \ref{dfggg-met} and later analyzed in Proposition \ref{asad}. 
It is worth noting that the notion of $(\phi,{\mathrm R}, {\mathcal B},{\mathbb F},{\mathcal P})$-normality generalizes many other notions of normality known in the existing literature (see, e.g., \cite[Definition 3.3, 4.2, 4.5, 5.17]{deda} for the one-dimensional setting, and \cite[Subsection 6.3.1]{nova-selected} for the multi-dimensional setting).
Definition \ref{nafaks-met}
introduces the notions of the Bohr $(\phi,{\mathbb F},{\mathcal B},\Lambda',\rho,{\mathcal P})$-almost periodicity and the $(\phi,{\mathbb F},{\mathcal B},\Lambda',\rho,{\mathcal P})$-uniform recurrence. 
The main aim of Proposition \ref{rep} is to indicate that the notion of $(\phi,{\mathrm R}, {\mathcal B},\phi,{\mathbb F},{\mathcal P})$-normality is very general as well as that any function belonging to the introduced function space 
$e-({\mathcal B},\phi,{\mathbb F})-B^{{\mathcal P}_{\cdot}}(\Lambda \times X : Y)$ [$e-({\mathcal B},\phi,{\mathbb F})_{j\in {\mathbb N}_{n}}-B^{{\mathcal P}_{\cdot}}_{{\rm I}}(\Lambda \times X : Y)$] is
$(\phi,{\mathrm R}, {\mathcal B},\phi,{\mathbb F},{\mathcal P})$-normal under certain logical assumptions (cf. also Proposition \ref{nbm}, where we reconsider the statements of \cite[Proposition 3.7, Corollary 3.8]{metrical} in our new framework). 

Without going into full description of introduced definitions and established results in Section \ref{general-stepa}, we will only emphasize that we investigate the generalized 
Stepanov, Weyl, Besicovitch and Doss concepts
to the metrical approximations of functions here. The Stepanov and Weyl concepts are investigated in Subsection \ref{anadol}, while the Besicovitch and Doss concepts are investigated in Subsection \ref{jet}. Several theoretical results 
about the convolution invariance of function spaces introduced in this paper, 
and many new
applications to the abstract Volterra integro-differential equations are given in Section \ref{maref} (the actions of infinite convolution products are specifically analyzed in Subsection \ref{seka}).
The final section of paper is reserved for the concluding remarks and observations about the introduced classes of functions.
\vspace{1.6pt}

\noindent {\bf Notation and terminology.} Suppose that $X,\ Y,\ Z$ and $ T$ are given non-empty sets. Let us recall that a binary relation between $X$ into $Y$
is any subset
$\rho \subseteq X \times Y.$ 
If $\rho \subseteq X\times Y$ and $\sigma \subseteq Z\times T$ with $Y \cap Z \neq \emptyset,$ then
we define
$\sigma \cdot  \rho =\sigma \circ \rho \subseteq X\times T$ by
$$
\sigma \circ \rho :=\bigl\{(x,t) \in X\times T : \exists y\in Y \cap Z\mbox{ such that }(x,y)\in \rho\mbox{ and }
(y,t)\in \sigma \bigr\}.
$$
As is well known, the domain and range of $\rho$ are defined by $D(\rho):=\{x\in X :
\exists y\in Y\mbox{ such that }(x,y)\in X\times Y \}$ and $R(\rho):=\{y\in Y :
\exists x\in X\mbox{ such that }(x,y)\in X\times Y\},$ respectively; $\rho (x):=\{y\in Y : (x,y)\in \rho\}$ ($x\in X$), $ x\ \rho \ y \Leftrightarrow (x,y)\in \rho .$
If $\rho$ is a binary relation on $X$ and $n\in {\mathbb N},$ then we define $\rho^{n}
$ inductively. Set $\rho (X'):=\{y : y\in \rho(x)\mbox{ for some }x\in X'\}$ ($X'\subseteq X$).

We will always assume henceforth that $(X,\| \cdot \|)$, $(Y, \|\cdot\|_Y)$ and $(Z, \|\cdot\|_Z)$ are three complex Banach spaces, $n\in {\mathbb N},$ $\emptyset  \neq \Lambda \subseteq {\mathbb R}^{n},$ and
${\mathcal B}$ is a non-empty collection of non-empty subsets of $X$ satisfying
that
for each $x\in X$ there exists $B\in {\mathcal B}$ such that $x\in B.$ For the sequel, we set
$$
\Lambda'':=\bigl\{ \tau \in {\mathbb R}^{n} : \tau +\Lambda \subseteq \Lambda \bigr\}.
$$
By
$L(X,Y)$ we denote the Banach space of all bounded linear operators from $X$ into
$Y,$ $L(X,X)\equiv L(X);$ ${\mathrm I}$ denotes the identity operator on $Y.$ The Lebesgue measure in ${\mathbb R}^{n}$ is denoted by $m(\cdot),$ and
the Wright function of order $\gamma \in (0,1)$ is denoted by
$\Phi_{\gamma}(\cdot);$ see, e.g., \cite{nova-mono} for the notion. 
Define 
${\mathbb N}_{n}:=\{1,..., n\};$ the standard basis of ${\mathbb R}^{n}$ is denoted by $(e_{1},...,e_{n}).$ If ${\mathrm A}$ and ${\mathrm B}$ are non-empty sets, then we define ${\mathrm B}^{{\mathrm A}}:=\{ f | f : {\mathrm A} \rightarrow {\mathrm B}\}.$ By $\| \cdot \|_{\infty}$ we denote the sup-norm; set also $\lceil s\rceil:=\inf\{ k\in {\mathbb Z} : s\leq k\}$ and 
$\lfloor s\rfloor:=\sup\{ k\in {\mathbb Z} : s\geq k\}$
for any $s\in {\mathbb R}.$

In this paper, we use the weighted function spaces (for further information concerning the Lebesgue spaces with variable exponents
$L^{p(x)},$ we refer the reader to \cite{variable}, \cite{nova-selected} and references cited therein).
Suppose first that the set $\Lambda$ is Lebesgue measurable and
$\nu : \Lambda \rightarrow (0,\infty)$ is a Lebesgue measurable function. By $ {\mathcal P}(\Lambda)$  we denote the space of all Lebesgue measurable functions from $\Lambda$ into $[1,+\infty]$.  In our analysis, we employ the following Banach space
$$
L^{p({\bf t})}_{\nu}(\Lambda: Y):=\bigl\{ u : \Lambda \rightarrow Y \ ; \ u(\cdot) \mbox{ is measurable and } ||u||_{p({\bf t})} <\infty \bigr\},
$$
where $p\in {\mathcal P}(\Lambda)$ and 
$$
\bigl\|u\bigr\|_{p({\bf t})}:=\bigl\| u({\bf t})\nu({\bf t}) \bigr\|_{L^{p({\bf t})}(\Lambda:Y)}.
$$
Suppose now that $\nu : \Lambda\rightarrow (0,\infty)$ is an arbitrary function satisfying that the function $1/\nu(\cdot)$ is locally bounded. Then the vector space $C_{0,\nu}(\Lambda : Y)$ [$C_{b,\nu}(\Lambda : Y)$] consists of all continuous functions $u : \Lambda \rightarrow
Y$ such that $\lim_{|{\bf t}|\rightarrow \infty , {\bf t}\in \Lambda}
\|u({\bf t})\|_{Y}\nu({\bf t})=0$ [$\sup_{{\bf t}\in \Lambda}\|u({\bf t})\|_{Y}\nu({\bf t})<+\infty$]. Equipped with the norm
$\|\cdot\|:=\sup _{{\bf t}\in \Lambda}\|\cdot({\bf t})\nu({\bf t})\|_{Y},$ $C_{0,\nu}(\Lambda : Y)$ [$C_{b,\nu}(\Lambda : Y)$] is a Banach space. Albeit it may look a little bit redundant, we have decided to repeat certain conditions on the function $\phi(\cdot)$ sometimes for the sake of better readability.

Finally, we need to recall the following notion:

\begin{defn}\label{drasko-presing}
\begin{itemize}
\item[(i)]
Let ${\bf \omega}\in {\mathbb R}^{n} \setminus \{0\},$ $\rho$ be a binary relation on $Y$ 
and 
${\bf \omega}\in \Lambda''$. A continuous
function $F:\Lambda \times X\rightarrow Y$ is said to be $({\bf \omega},\rho)$-periodic [$\rho$-periodic] if and only if 
$
F({\bf t}+{\bf \omega};x)\in \rho(F({\bf t};x)),$ ${\bf t}\in \Lambda,$ $x\in X$ [there exists ${\bf \omega}\in ({\mathbb R}^{n} \setminus \{0\}) \cap \Lambda''$ such that $F(\cdot;\cdot)$ is $({\bf \omega},\rho)$-periodic]. 
\item[(ii)]
Let ${\bf \omega}_{j}\in {\mathbb R} \setminus \{0\},$ $\rho_{j}\in {\mathbb C} \setminus \{0\}$ be a binary relation on $Y$
and 
${\bf \omega}_{j}e_{j}+\Lambda \subseteq \Lambda$ ($1\leq j\leq n$). A continuous
function $F:\Lambda \times X \rightarrow Y$ is said to be $({\bf \omega}_{j},\rho_{j})_{j\in {\mathbb N}_{n}}$-periodic [$(\rho_{j})_{j\in {\mathbb N}_{n}}$-periodic] if and only if 
$
F({\bf t}+{\bf \omega}_{j}e_{j};x)\in \rho_{j}(F({\bf t};x)),$ ${\bf t}\in \Lambda,
$ $x\in X,$ $j\in {\mathbb N}_{n}$ [there exist non-zero real numbers $\omega_{j}$ such that ${\bf \omega}_{j}e_{j}\in \Lambda''$ for all $j\in {\mathbb N}_{n}$ and $F(\cdot;\cdot)$ is $({\bf \omega}_{j},\rho_{j})_{j\in {\mathbb N}_{n}}$-periodic].
\item[(iii)] Let ${\bf \omega}_{j}\in {\mathbb R} \setminus \{0\},$ 
and 
${\bf \omega}_{j}e_{j}+\Lambda \subseteq \Lambda$ ($1\leq j\leq n$). A continuous
function $F:\Lambda \times X \rightarrow Y$ is said to be periodic if and only if $F(\cdot;\cdot)$ is $(\rho_{j})_{j\in {\mathbb N}_{n}}$-periodic with $\rho_{j}={\rm I}$ for all $j\in {\mathbb N}_{n}.$
\end{itemize}
\end{defn} \index{function!$({\bf \omega}_{j},\rho_{j})_{j\in {\mathbb N}_{n}}$-periodic}

By a trigonometric polynomial $P : \Lambda \times X \rightarrow Y$ we mean any linear combination of functions like
\begin{align*}
e^{i[\lambda_{1}t_{1}+\lambda_{2}t_{2}+\cdot \cdot \cdot +\lambda_{n}t_{n}]}c(x),
\end{align*}
where $\lambda_{i}$ are real numbers ($1\leq i \leq n$) and $c: X \rightarrow Y$ is a continuous mapping.

\section{Metrical approximations: the main concept}\label{docolord}

In this section, we assume that $\emptyset \neq \Lambda \subseteq {\mathbb R}^{n},$ 
$\phi : [0,\infty) \rightarrow [0,\infty),$ ${\mathbb F} : \Lambda \rightarrow (0,\infty)$,
$P \subseteq [0,\infty)^{\Lambda},$ the space of all functions from $\Lambda$ into $[0,\infty),$  the zero function belongs to $P$, and 
${\mathcal P}=(P,d)$ is a pseudometric space.

We start by introducing the following notion:

\begin{defn}\label{strong-app} 
Suppose that $\emptyset  \neq \Lambda \subseteq {\mathbb R}^{n}$ and $F :\Lambda  \times X \rightarrow Y.$
Then we say that $F(\cdot;\cdot)$ is strongly $(\phi,{\mathbb F},{\mathcal B},{\mathcal P})$-almost periodic (semi-$(\phi,\rho,{\mathbb F},{\mathcal B},{\mathcal P})$-periodic, semi-$(\phi,\rho_{j},{\mathbb F},{\mathcal B},{\mathcal P})_{j\in {\mathbb N}_{n}}$-periodic) if and only if for each $B\in {\mathcal B}$ there exists a sequence $(P_{k}^{B}({\bf t};x))$ of trigonometric polynomials ($\rho$-periodic functions, $(\rho_{j})_{j\in {\mathbb N}_{n}}$-periodic functions)
such that 
\begin{align}\label{gnarbi}
\lim_{k\rightarrow +\infty}\sup_{x\in B} \Bigl\|{\mathbb F}(\cdot) \phi\Bigl(\bigl\|P_{k}^{B}(\cdot;x)-F(\cdot;x)\bigr\|_{Y}\Bigr)\Bigr\|_{P}=0.
\end{align}
\end{defn}\index{function!strongly $(\phi,{\mathbb F},{\mathcal B},{\mathcal P})$-almost periodic} \index{function!semi-$(\phi,\rho,{\mathbb F},{\mathcal B},{\mathcal P})$-periodic}\index{function!semi-$(\phi,\rho_{j},{\mathbb F},{\mathcal B},{\mathcal P})_{j\in {\mathbb N}_{n}}$-periodic}

The usual notion of a strongly ${\mathcal B}$-almost periodic function $F : \Lambda \rightarrow Y$ (semi-$(\rho_{j})_{j\in {\mathbb N}_{n}}$-periodic function $F : \Lambda \rightarrow Y$) is obtained by plugging ${\mathbb F}(\cdot)\equiv 1,$ $\phi(x) \equiv x$ and $P=C_{b}(\Lambda :Y).$ 

Our first result reads as follows:

\begin{prop}\label{kimih}
Suppose that $\emptyset  \neq \Lambda \subseteq {\mathbb R}^{n}$, $F :\Lambda  \times X \rightarrow Y$, $h : Y \rightarrow Z$ is Lipschitz continuous, $\phi(\cdot)$ is monotonically increasing and there exists a function $\varphi : [0,\infty) \rightarrow [0,\infty)$ such that $\phi(xy)\leq \varphi(x)\phi(y)$ for all $x,\ y\geq 0.$ Let the assumption \emph{(C1)} hold, where:\index{condition!(C1)}\index{condition!(C2)}
\begin{itemize}
\item[(C1)]
If
$f\in P,$ then $d'f\in P$ for all reals $d'\geq 0,$ and there exists a finite real constant $d>0$ such that $\|d'f\|_{P}\leq d(1+d')\|f\|_{P}$ for all reals $d'\geq 0$ and all functions $f\in P.$
\end{itemize}
Then we have the following:
\begin{itemize}
\item[(i)]
Suppose that $F(\cdot;\cdot)$ is semi-$(\phi,\rho,{\mathbb F},{\mathcal B},{\mathcal P})$-periodic (semi-$(\phi,\rho_{j},{\mathbb F},{\mathcal B},{\mathcal P})_{j\in {\mathbb N}_{n}}$-periodic), and
$h\circ \rho \subseteq \rho \circ h$ ($h\circ \rho_{j} \subseteq \rho_{j} \circ h$ for $1\leq j\leq n$).
Then the function $h\circ F : \Lambda \times X \rightarrow Y$ is likewise semi-$(\phi,\rho,{\mathbb F},{\mathcal B},{\mathcal P})$-periodic (semi-$(\phi,\rho_{j},{\mathbb F},{\mathcal B},{\mathcal P})_{j\in {\mathbb N}_{n}}$-periodic). 
\item[(ii)] Suppose that $X=\{0\}$,  there exists a finite real constant $c>0$ such that $\phi(x+y)\leq c[\varphi(x)+\varphi(y)]$ for all $x,\ y\geq 0,$ $\phi(\cdot)$ is continuous at the point zero, 
\begin{align}\label{mitko}
{\mathbb F}\in P\ \ \mbox{ and }\lim_{\epsilon \rightarrow 0+}\|\epsilon {\mathbb F}(\cdot)\|_{P}=0,
\end{align}
the function $F(\cdot)$ is strongly $(\phi,{\mathbb F},{\mathcal B},{\mathcal P})$-almost periodic and 
the assumption \emph{(C2)} holds, where:
\begin{itemize}
\item[(C2)]
There exists a finite real constant $e>0$ such that the assumptions $f,\ g\in P$ and $0\leq w\leq d'[f+g]$ for some finite real constant $d'>0$ imply $w\in P$ and $\| w\|_{P}\leq e(1+d')[\|f\|_{P}+\|g\|_{P}].$ 
\end{itemize}
Then the function $(h\circ F)(\cdot)$ is likewise strongly $(\phi,{\mathbb F},{\mathcal B},{\mathcal P})$-almost periodic.
\end{itemize}
\end{prop}

\begin{proof}
In order to prove (i), observe first that, if $P:\Lambda \times X\rightarrow Y$ is $({\bf \omega},\rho)$-periodic ($({\bf \omega}_{j},\rho_{j})_{j\in {\mathbb N}_{n}}$-periodic), then the function $h\circ P : \Lambda \times X\rightarrow Z$ is likewise $({\bf \omega},\rho)$-periodic ($({\bf \omega}_{j},\rho_{j})_{j\in {\mathbb N}_{n}}$-periodic) since we have assumed that $h\circ \rho \subseteq \rho \circ h$ ($h\circ \rho_{j} \subseteq \rho_{j} \circ h$ for $1\leq j\leq n$). Let $L>0$ be the Lipschitzian constant of mapping $h(\cdot).$ Then the required statement simply follows from our assumptions on the function $\phi(\cdot),$ the pseudometric space $P$ and the subsequent computation (the set $B\in {\mathcal B}$ is given in advance):
\begin{align*}
 \Bigl\|{\mathbb F}(\cdot)& \phi\Bigl(\bigl\|\bigl(h\circ P_{k}^{B}\bigr)(\cdot;x)-\bigl( h\circ F(\cdot;x)\bigr)\bigr\|_{Y}\Bigr)\Bigr\|_{P}
\\& \leq  \Bigl\|{\mathbb F}(\cdot) \phi\Bigl(L\bigl\|P_{k}^{B}(\cdot;x)-F(\cdot;x)\bigr\|_{Y}\Bigr)\Bigr\|_{P}
\\& \leq \Bigl\|{\mathbb F}(\cdot) \varphi(L)\phi\Bigl(\bigl\|P_{k}^{B}(\cdot;x)-F(\cdot;x)\bigr\|_{Y}\Bigr)\Bigr\|_{P}
\\& \leq d(1+\varphi(L))\Bigl\|{\mathbb F}(\cdot) \phi\Bigl(\bigl\|P_{k}^{B}(\cdot;x)-F(\cdot;x)\bigr\|_{Y}\Bigr)\Bigr\|_{P}.
\end{align*}
To prove (ii), observe first that the function $(h\circ P)(\cdot)$ is strongly almost periodic for every trigonometric polynomial $P(\cdot);$ hence, there exists a sequence $(P_{k})$ of trigonometric polynomials such that $\lim_{k\rightarrow \infty}\|P_{k}-(h\circ P)\|_{\infty}=0$ (\cite{nova-selected}). After that, we can apply the argumentation used for proving (i), the additional assumptions given and the next computation:{\small
\begin{align*}
& \Bigl\|{\mathbb F}(\cdot) \phi\Bigl(\bigl\|P_{k}(\cdot)-\bigl( h\circ F\bigr)(\cdot)\bigr\|_{Y}\Bigr)\Bigr\|_{P}
\\& \leq e(1+c)\Biggl[\Bigl\|{\mathbb F}(\cdot) \phi\Bigl(\bigl\|P_{k}(\cdot)-\bigl( h\circ P\bigr)(\cdot)\bigr\|_{Y}\Bigr)\Bigr\|_{P}+
\Bigl\|{\mathbb F}(\cdot) \phi\Bigl(\bigl\|\bigl( h\circ P\bigr)(\cdot)-\bigl( h\circ F\bigr)(\cdot)\bigr\|_{Y}\Bigr)\Bigr\|_{P}\Biggr]
\\& \leq e(1+c)\Biggl[\Bigl\|{\mathbb F}(\cdot) \phi\bigl(\epsilon'\bigr)\Bigr\|_{P}+d(1+\varphi(L))
\Bigl\|{\mathbb F}(\cdot) \phi\Bigl(\bigl\|P(\cdot)- F(\cdot)\bigr\|_{Y}\Bigr)\Bigr\|_{P}\Biggr];
\end{align*}}
here, we first choose a finite real number $\epsilon_{0}>0$ such that $\|x {\mathbb F}(\cdot)\|_{P}<\epsilon/(2e(1+c))$ for $0<x<\epsilon_{0}$ (cf. \eqref{mitko}),
and after that we choose a finite real number
$\epsilon'>0$ such that $ \phi(\epsilon')<\epsilon_{0}$ due to the continuity of function $\phi(\cdot)$ at the point zero.
\end{proof}

Without going into further details, we will only note here that Proposition \ref{kimih} can be reformulated for all other classes of functions analyzed in this paper. Condition (C1) holds if $P$ is a Banach space or $P$ is a Fr\' echet space and $d(\cdot;\cdot)$ is a metric induced by the fundamental system of increasing seminorms which defines the topology of $P;$ note also that condition \eqref{mitko} does not hold in general metric spaces $P$: for example, if $d(\cdot;\cdot)$ is a
discrete unit metric on $P,$ defined by $d(f,g):=0$ if and only if $f=g,$ and $d(f,g):=1,$ otherwise, then \eqref{mitko} does not hold. With this choice of metric space $P$, we have that a function $F(\cdot;\cdot)$ is strongly $(x,1,{\mathcal B},{\mathcal P})$-almost periodic (semi-$(x,\rho,1,{\mathcal B},{\mathcal P})$-periodic, semi-$(x,\rho_{j},1,{\mathcal B},{\mathcal P})_{j\in {\mathbb N}_{n}}$-periodic) if and only if $F(\cdot;\cdot)$ is a trigonometric polynomial ($\rho$-periodic function, $(\rho_{j})_{j\in {\mathbb N}_{n}}$-periodic function).

We continue by providing several illustrative examples:

\begin{example}\label{ferplay}
\begin{itemize}
\item[(i)] The spaces of Bohr almost periodic functions (semi-$(c_{j})_{j\in {\mathbb N}_{n}}$-periodic functions; cf. \cite[Subsection 2.1]{brno} and \cite[Section 2]{chaouchi}) $F: {\mathbb R}^{n} \rightarrow Y$ can be constricted if we use the pseudometric spaces ${\mathcal P}$ such that ${\mathcal P}$ is continuously embedded into the space $C_{b}({\mathbb R}^{n} : Y)$. The obvious choice is the space $C_{b,\nu}({\mathbb R}^{n} : Y)$, where the function $1/\nu(\cdot)$ is locally bounded
and
$\nu(x)\geq c,$ $x\in {\mathbb R}^{n}$ for some positive real number $c>0$. 

Concerning the choice of this space, we will present the following example appearing in \cite{oliaro}. If $\emptyset \neq \Omega \subseteq {\mathbb R}^{n}$, then the space of all Gevrey functions of order $s\geq 1,$ denoted by $G^{s}(\Omega),$ is defined as a collection of all infinitely differentiable functions $F : {\mathbb R}^{n}\rightarrow {\mathbb C}$ such that for each compact set $K\subseteq {\mathbb R}^{n}$ there exists a finite real constant $C_{K}>0$ such that $$|D^{\alpha}F({\bf t})| \leq C_{K}^{1+|\alpha|}\alpha!^{s}$$ for all ${\bf t}\in K$ and $\alpha \in {\mathbb N}_{0}^{n}.$ It is natural to ask whether an almost periodic function $F : {\mathbb R}^{n}\rightarrow {\mathbb C}$ which belongs to the space $G^{s}(\Omega)$ obeys the property of the existence of a global real constant $C>0$ such that  $$|D^{\alpha}F({\bf t})| \leq C^{1+|\alpha|}\alpha!^{s}$$ for all ${\bf t}\in {\mathbb R}^{n}$ and $\alpha \in {\mathbb N}_{0}^{n}?$
An instructive counterexample in the one-dimensional setting, with $s>1,$ is given in \cite[Example 2.1]{oliaro}, showing that this is not true in general. We will reexamine this example for our purposes. 

Set $g_{s}(x):=\exp(-x^{1/(1-s)}),$ $x>0,$ $g_{s}(x):=0,$ $x\leq 0,$
$\psi_{s}(x):=g_{s}(x)g_{s}(1-x),$ $x\in {\mathbb R},$ $\psi_{s,n}(x):=\psi_{s}(nx),$ $x\in {\mathbb R}$ and $\varphi_{s,n}(x):=\sum_{k\in {\mathbb Z}}\psi_{s}(x-2^{n}(2k+1)), $ $x\in {\mathbb R}$ ($n\in {\mathbb N}$). It has been shown that the function
$$
F_{s}(x):=\sum_{k=1}^{\infty}k^{-1/4}\varphi_{s,k}(x),\quad x\in {\mathbb R}
$$
is well defined, as well as that the above series is uniformly convergent in the variable $x\in {\mathbb R},$ so that the function $F_{s}(\cdot)$ is actually, semi-periodic, since the function $\varphi_{s,n}(\cdot)$ is of period $2^{n+1}$ ($n\in {\mathbb N}$). Moreover, it has been shown that for each $x\in {\mathbb R}$ and $N\in {\mathbb N}$ we have the existence of an integer $n_{0}>N$ such that 
\begin{align}\label{tacan}
\sum_{k>N}n^{-1/4}\varphi_{s,k}(x)=n_{0}^{-1/4}\varphi_{s,n_{0}}(x)\leq n_{0}^{-1/4}\leq (1+N)^{-1/4}.
\end{align}
Let $(x_{n_{0}})$ be any strictly increasing sequence tending to plus infinity and satisfying that $x_{n_{0}}\in \bigcup_{k\in {\mathbb Z}}([0,n_{0}^{-1}] +2^{n_{0}}(1+2k)).$ 
Define the function $\nu : {\mathbb R}\rightarrow (0,\infty)$ by $\nu(x):=1,$ if $x\notin \{x_{n_{0}} : n_{0}\in {\mathbb N}\}$ and $\nu(x_{n_{0}}):=n_{0}^{1/8}.$ The use of \eqref{tacan} implies that the function $F_{s}(\cdot)$ is semi-$(x,{\rm I},1,{\mathcal P})$-periodic with $P=C_{b,\nu}({\mathbb R} : [0,\infty)).$

It is worth noting that we can also use the metric spaces ${\mathcal P}$ equipped with the distance of the form 
\begin{align}\label{1258}
d(f,g)=\|f-g\|_{\infty}+d_{1}(f,g),\quad f,\ g\in P,
\end{align}
where $P$ is a certain subspace of the space $C_{b}({\mathbb R}^{n} : [0,\infty))$ and $d_{1}(\cdot;\cdot)$ is a pseudometric on $P.$ In the one-dimensional setting, this approach has been used in many research articles of S. Sto\' inski and his followers (see, e.g., \cite{stoinski2}-\cite{fasc}). 
Concerning this problematic, we will first present here an example of a $p$-semi-anti-periodic function in variation ($1\leq p<+\infty$) based on our analysis from \cite[Example 4.2.9]{nova-selected} and the investigations carried out in \cite{stoja} and \cite{fasc}.
We already know that the function
$$
f(x):=\sum_{m=1}^{\infty}\frac{e^{ix/(2m+1)}}{m^{2}},\quad x\in {\mathbb R}
$$
is not periodic and $f(\cdot)$ is semi-anti-periodic ($n=1,$ $c_{1}=-1$) because it is a uniform limit of $[\pi \cdot (2N+1)!!]$-anti-periodic functions
$$
f_{N}(x):=\sum_{m=1}^{N}\frac{e^{ix/(2m+1)}}{m^{2}},\quad x\in {\mathbb R} \ \ (N\in {\mathbb N});
$$
cf. also \cite[Example 2.11]{brno} for the multi-dimensional analogue of this example.
Now we will prove that $f(\cdot)$ is $p$-semi-anti-periodic function in variation ($1\leq p<+\infty$), i.e., semi-$(x,-{\rm I},1,{\mathcal P})$-periodic with $P$ being the subspace $BV_{p}({\mathbb R} : [0,\infty))$ of $C_{b}({\mathbb R} : [0,\infty))$ consisting of those functions $h(\cdot)$ for which the $p$-variation of $h(\cdot)$ on the interval $[t-1,t+1],$ defined by\index{function!$p$-semi-anti-periodic function in variation}
$$
V_{p}(h;t):=\sup_{\Phi}\Biggl( \sum_{i=0}^{s-1}\bigl| h(u_{i+1})-h(u_{i})  \bigr|^{p} \Biggr)^{1/p},
$$
is finite for any $t\in {\mathbb R}$ (the supremum is taken over all finite partitions $\Phi=\{u_{0},...,u_{s}\},$ $t-1=u_{0}<u_{1}<...<u_{s}=t+1,$ of the interval $[t-1,t+1]$)
and $\sup_{t\in {\mathbb R}}V_{p}(f;t)<+\infty.$ We equip the space $BV_{p}({\mathbb R} : [0,\infty))$ with the metric
\begin{align}\label{12589}
d(f,g):=\sup_{t\in {\mathbb R}}\Bigl(|f(t)-g(t)|+V_{p}(f-g;t)\Bigr),\quad f,\ g\in BV_{p}({\mathbb R} : [0,\infty));
\end{align}
cf. also \eqref{1258}.
Using the partial sums $f_{N}(\cdot),$ it suffices to show that
\begin{align}\label{varir}
\lim_{N\rightarrow +\infty}\sup_{t\in {\mathbb R}}V_{p}\Biggl( \Biggl| \sum_{m=N+1}^{\infty}\frac{e^{i\cdot/(2m+1)}}{m^{2}} \Biggr| ;t\Biggr)=0.
\end{align}
Let $t\in {\mathbb R}$ be fixed, and let $\Phi=\{u_{0},...,u_{s}\},$ $t-1=u_{0}<u_{1}<...<u_{s}=t+1,$ be any finite partition of the interval $[t-1,t+1].$ Then we have:
\begin{align*}
\sup_{\Phi}&\Biggl( \sum_{i=0}^{s-1}\Biggl| \Biggl| \sum_{m=N+1}^{\infty}\frac{e^{iu_{i+1}/(2m+1)}}{m^{2}} \Biggr| -  \Biggl| \sum_{m=N+1}^{\infty}\frac{e^{iu_{i}/(2m+1)}}{m^{2}} \Biggr|\Biggr|^{p} \Biggr)^{1/p}
\\& \leq \sup_{\Phi}\Biggl( \sum_{i=0}^{s-1}\Biggl|  \sum_{m=N+1}^{\infty}\frac{e^{iu_{i+1}/(2m+1)}}{m^{2}} - \sum_{m=N+1}^{\infty}\frac{e^{iu_{i}/(2m+1)}}{m^{2}} \Biggr|^{p} \Biggr)^{1/p}
\\& \leq \sup_{\Phi}\Biggl( \sum_{i=0}^{s-1}\Biggl|  \sum_{m=N+1}^{\infty}\frac{\bigl|e^{iu_{i+1}/(2m+1)}-e^{iu_{i}/(2m+1)}\bigr|}{m^{2}} \Biggr|^{p} \Biggr)^{1/p}
\\& \leq \sup_{\Phi}\Biggl( \sum_{i=0}^{s-1}\Biggl|  \sum_{m=N+1}^{\infty}\frac{\bigl|u_{i+1}-u_{i}\bigr|}{m^{2}(2m+1)} \Biggr|^{p} \Biggr)^{1/p}
\\& \leq \sup_{\Phi}\Biggl( \sum_{i=0}^{s-1}\bigl| u_{i+1}-u_{i}  \bigr|^{p} \Biggr)^{1/p}\sum_{m=N+1}^{\infty}\frac{1}{m^{2}(2m+1)}
\\& \leq \sup_{\Phi} \sum_{i=0}^{s-1}\bigl| u_{i+1}-u_{i}  \bigr|\sum_{m=N+1}^{\infty}\frac{1}{m^{2}(2m+1)}
\\&=2\sum_{m=N+1}^{\infty}\frac{1}{m^{2}(2m+1)}
\rightarrow 0,\quad N\rightarrow +\infty.
\end{align*}
\item[(ii)]
If we remove the term $|f(t)-g(t)|$ in \eqref{12589}, then we obtain the concept of slow $p$-semi-periodicity in variation ($1\leq p<+\infty$); see also \cite[pp. 563-564, Theorem 6]{stoj-nov}. For the sequel, let us recall that A. Haraux and P. Souplet have proved, in \cite[Theorem 1.1]{haraux}, that the function
$$
f(t):=\sum_{m=1}^{\infty}\frac{1}{m}\sin^{2}\Bigl( \frac{t}{2^{m}} \Bigr),\quad t\in {\mathbb R},
$$\index{function!slowly $p$-semi-periodic in variation}
is not Besicovitch-$p$-almost periodic for any finite exponent $p\geq 1$, as well as that $f(\cdot)$ is uniformly recurrent and uniformly continuous;  see \cite{nova-mono} for the notion and more details.
Now we will prove that the function $f(\cdot)$ is slowly $p$-semi-periodic in variation ($1\leq p<+\infty$). In our approach,
we consider the pseudometric space ${\mathcal P}:=(P,d_{1}),$ where $P=BV_{p}({\mathbb R} : [0,\infty))$ 
and $d_{1}(f,g):=\sup_{t\in {\mathbb R}}V_{p}(f-g;t),$ $f,\ g\in P;$ then we actually want to show that the function  
$f(\cdot)$ is $(x,1,{\mathcal P})$-semi-periodic; but, this simply follows from the equality\index{space!$BV_{p}({\mathbb R} : [0,\infty))$}
\begin{align*}
\lim_{N\rightarrow +\infty}\sup_{t\in {\mathbb R}}V_{p}\Biggl( \Biggl| \sum_{m=N+1}^{\infty}\frac{\sin^{2}(\cdot/2^{m})}{m} \Biggr| ;t\Biggr)=0,
\end{align*}
which can be proved as in part (i), for the equality \eqref{varir}.
\item[(iii)] The spaces of Bohr almost periodic functions  (semi-$(c_{j})_{j\in {\mathbb N}_{n}}$-periodic functions) $F: {\mathbb R}^{n} \rightarrow Y$ can be extended if we use the pseudometric spaces ${\mathcal P}$ such that the space $C_{b}({\mathbb R}^{n} : Y)$ is continuously embedded into ${\mathcal P}.$ The obvious choice is the space $C_{b,\nu}({\mathbb R}^{n} : Y)$, where the function $1/\nu(\cdot)$ is locally bounded
and
$\nu(x)\leq c,$ $x\in {\mathbb R}^{n}$ for some finite real number $c>0.$ But, here we can also use the metric space $P$ consisting of all continuous functions from ${\mathbb R}^{n}$ into $[0,\infty),$ equipped with the distance
$$
d(f,g):=\sup_{x\in {\mathbb R}^{n}}\bigl| \arctan (f(x))-\arctan(g(x)) \bigr|,\quad f,\ g\in P.
$$
Concerning the use of this metric space, we would like to note first that ${\mathcal P}$ is not complete since the sequence 
$$
\Bigl(f_{n}(\cdot):=\tan \bigl((\pi/2)-[\cos^{2}\cdot]-(1/n)\bigr)\Bigr)_{n\in {\mathbb N}}
$$ 
of periodic functions is a Cauchy sequence in ${\mathcal P}$ but not convergent.  
Denote by $\overline{{\mathcal P}}=(\overline{P},\overline{d})$
the completion of ${\mathcal P}=(P,d).$ Then the limit function $f(\cdot)$ of sequence $(f_{n}(\cdot))$ has the form $f(x)=\tan((\pi/2)-[\cos^{2}x])=\cot (\cos^{2}x),$ $x\notin (\pi/2)+{\mathbb Z}\pi$ and $f(x)=+\infty,$ if $x\in (\pi/2)+{\mathbb Z}\pi$. Therefore, the function $f(\cdot)$ is not locally integrable since
$$
f(x)\sim \frac{1}{\cos^{2}x} \sim \frac{1}{((\pi/2)+k\pi-x)^{2}},\quad x\rightarrow (\pi/2)+k\pi \ \ (k\in {\mathbb Z}),
$$
which implies that we cannot expect that the function $f(\cdot)$ is Besicovitch (Weyl, Stepanov) almost periodic in the usual sense (\cite{nova-selected}). On the other hand, it is clear that we can extend the notion introduced in Definition \ref{strong-app}
by replacing the limit equality \eqref{gnarbi} by 
\begin{align}\label{pehar}
\lim_{k\rightarrow +\infty}\sup_{x\in B} \Bigl\|{\mathbb F}(\cdot) \phi\Bigl(\bigl\|P_{k}^{B}(\cdot;x)-F(\cdot;x)\bigr\|_{Y}\Bigr)\Bigr\|_{\overline{P}}=0.
\end{align}
Then the function $f(\cdot)$ considered above will be $(x,1,\overline{{\mathcal P}})$-semi-periodic, with the meaning clear. We will further analyze the  approximation equality \eqref{pehar} in our forthcoming research studies; this concept seems to be very important because it allows one to consider, maybe for the first time in existing literature, the generalized almost periodicity of functions 
 $f(\cdot)$ which are not locally integrable.\index{metric space!completion}
\item[(iv)] The use of a complete metric space $P=C({\mathbb R}^{n} )$ is completely irrelevant in our analysis (the use of function spaces $P=C^{k}({\mathbb R}^{n} )$, where $k\in {\mathbb N},$ $P=C^{\infty}({\mathbb R}^{n} )$ and some ultradistributional analogues of these spaces is much more important but we will not consider here this topic). More precisely, let $P$ be equipped with the metric
$$
d(f,g):=\sum_{k=1}^{\infty}2^{-k}\frac{\sup_{x\in [-k,k]^{n}}| f(x)-g(x)|}{1+\sup_{x\in [-k,k]^{n}}| f(x)-g(x)|},\quad f,\ g\in P,
$$
and let $\phi(x)\equiv x,$ ${\mathbb F}\equiv 1.$ Then a continuous function $F : {\mathbb R}^{n} \rightarrow Y$ is $(x,1,{\mathcal P})$-almost periodic if and only if there exists a sequence $(P_{k})$ of trigonometric polynomials which converges uniformly to $F(\cdot)$ on any compact subset of ${\mathbb R}^{n}.$ 
Using the vector-valued version of the Weierstrass approximation theorem, we can simply prove that any continuous function $F : {\mathbb R}^{n} \rightarrow Y$ is $(x,1,{\mathcal P})$-almost periodic; in actual fact, for every integer $k\in {\mathbb N},$ we can find a trigonometric polynomial $P_{k} : {\mathbb R}^{n} \rightarrow Y$ such that $\sup_{|{\bf t}|\leq k}\| P_{k}({\bf t})-F({\bf t})\|_{Y} \leq 1/k.$ Then $(P_{k})$ converges uniformly to $F(\cdot)$ on any compact subset of ${\mathbb R}^{n}.$ 
\end{itemize}
\end{example}

\subsection{Metrical normality and metrical Bohr type definitions}\label{zeljo}

We refer the reader to the survey article \cite{deda} by J. Andres, A. M. Bersani, R. F. Grande and the monograph \cite{nova-selected} for more details about various classes of normal type functions in the theory of almost periodic functions.
We start this subsection by introducing the following notion, which generalizes the usual notion of Bohr normality (see, e.g., \cite[Definition 2.6]{deda} for the one-dimensional setting):

\begin{defn}\label{dfggg-met}
Suppose that ${\mathrm R}$ is any collection of sequences in $\Lambda'',$ $F: \Lambda \times X \rightarrow Y,$ $\phi : [0,\infty) \rightarrow [0,\infty)$ and ${\mathbb F} : \Lambda \rightarrow (0,\infty).$ Then we say that the function $F(\cdot;\cdot)$ is
$(\phi,{\mathrm R}, {\mathcal B},{\mathbb F},{\mathcal P})$-normal
if and only if for every set $B\in {\mathcal B}$ and for every sequence $({\bf b}_{k})_{k\in {\mathbb N}}$ in ${\mathrm R}$ there exists a subsequence $({\bf b}_{k_{m}})_{m\in {\mathbb N}}$ of $({\bf b}_{k})_{k\in {\mathbb N}}$ such that, for every $\epsilon>0,$ there exists an integer $m_{0}\in {\mathbb N}$ such that, for every integers $m,\ m'\geq m_{0},$ we have 
\begin{align}\label{nd}
\sup_{x\in B}\Biggl\|{\mathbb F}(\cdot)\phi\Bigl( \bigl\| F(\cdot+{\bf b}_{k_{m}};x)-F(\cdot+{\bf b}_{k_{m'}};x)\bigr\|_{Y} \Bigr)\Biggr\|_{P} <\epsilon.
\end{align}
\end{defn}\index{function!$(\phi,{\mathrm R}, {\mathcal B},{\mathbb F},{\mathcal P})$-normal}

The usual notion of $({\mathrm R}, {\mathcal B},{\mathcal P})$-normality is obtained by plugging ${\mathbb F}(\cdot)\equiv 1,$ $\phi(x)\equiv x,$ 
${\mathrm R}$ being a collection of sequences in $\Lambda'',$ and $P=L^{\infty}(\Lambda);$ see also \cite[Definition 2.1]{metrical}, where we have assumed that \eqref{nd} holds for every set $B'$ of a collection $L(B; {\bf b})$
of certain subsets of $B.$ The notion introduced in \cite[Definition 2.1, Definition 2.2]{metrical} can be further strengthened by using the general functions ${\mathbb F}(\cdot)$ and $\phi(x)$ therein. 

The uniformly convergent sequences of almost periodic type functions and the behaviour of its limit function have been examined in many structural results established so far (see, e.g., \cite[Proposition 2.7]{metrical} for such a result regarding metrically almost periodic type functions). In this paper, we will clarify only one result concerning this issue:

\begin{prop}\label{asad}
Suppose that ${\mathrm R}$ is any collection of sequences in $\Lambda'',$ $F_{j}: \Lambda \times X \rightarrow Y,$ $\phi : [0,\infty) \rightarrow [0,\infty),$ ${\mathbb F} : \Lambda \rightarrow (0,\infty),$ and the function $F_{j}(\cdot;\cdot)$ is
$(\phi,{\mathrm R}, {\mathcal B},{\mathbb F},{\mathcal P})$-normal for all $j\in {\mathbb N}.$ If $F: \Lambda \times X \rightarrow Y$ and, for every set $B\in {\mathcal B}$ and for every sequence  $({\bf b}_{k})_{k\in {\mathbb N}}$ in ${\mathrm R},$ we have
\begin{align}\label{zrenjanin}
\lim_{(j,k)\rightarrow +\infty}\sup_{x\in B}\Biggl\|{\mathbb F}(\cdot)\phi\Bigl(\bigl\| F_{j}(\cdot +{\bf b}_{j};x)-F(\cdot +{\bf b}_{k};x)\bigr\|_{Y}\Bigr) \Biggr\|_{P}=0,
\end{align}
then the function $F(\cdot;\cdot)$ is likewise $(\phi,{\mathrm R}, {\mathcal B},{\mathbb F},{\mathcal P})$-normal, provided that:
\begin{itemize}
\item[(i)] The function $\phi(\cdot)$ is monotonically increasing and there exists a finite real constant $c>0$ such that $\phi(x+y)\leq c[\phi(x)+\phi(y)]$ for all $x,\ y\geq 0.$
\item[(ii)] Condition \emph{(C3)} holds, where:
\begin{itemize}
\item[(C3)]
There exists a finite real constant $f>0$ such that the assumptions $f,\ g,\ h\in P$ and $0\leq w\leq d'[f+g+h]$ for some finite real constant $d'>0$ imply $w\in P$ and $\| w\|_{P}\leq f(1+d')[\|f\|_{P}+\|g\|_{P}+\|h\|_{P}].$
\end{itemize}
\end{itemize}
\end{prop}\index{condition!(C3)}

\begin{proof}
Without loss of generality, we may assume that $c=f=1.$
Let a real number $\epsilon>0$, a set $B\in {\mathcal B}$ and a sequence $({\bf b}_{k})_{k\in {\mathbb N}}$ in ${\mathrm R}$ be given. 
Then there exists a natural number $N\in {\mathbb N}$
such that the assumption $\min(j,k)\geq N$ implies 
$$
\sup_{x\in B}\Biggl\|{\mathbb F}(\cdot)\phi\Bigl(\bigl\| F_{j}(\cdot +{\bf b}_{j};x)-F(\cdot +{\bf b}_{k};x)\bigr\|_{Y}\Bigr) \Biggr\|_{P}<\epsilon/3.
$$
After that, we find a subsequence $({\bf b}_{k_{m}})_{m\in {\mathbb N}}$ of $({\bf b}_{k})_{k\in {\mathbb N}}$ and a natural number $m_{0}\in {\mathbb N}$ such that, for every integers $m,\ m'\geq m_{0},$ we have:
\begin{align*}
\sup_{x\in B}\Biggl\|{\mathbb F}(\cdot)\phi\Bigl( \bigl\| F_{N}(\cdot+{\bf b}_{k_{m}};x)-F_{N}(\cdot+{\bf b}_{k_{m'}};x)\bigr\|_{Y} \Bigr)\Biggr\|_{P} <\epsilon/3.
\end{align*}
Then the final conclusion simply follows from the last two estimates, conditions (i)-(ii) and the next decomposition ($m,\ m'\geq m_{0}$):
\begin{align*}
\sup_{x\in B}&\Biggl\|{\mathbb F}(\cdot)\phi\Bigl(\bigl\| F(\cdot +{\bf b}_{k_{m}};x)-F(\cdot +{\bf b}_{k_{m'}};x)\bigr\|_{Y}\Bigr) \Biggr\|_{P}
\\& \leq \sup_{x\in B}\Biggl\|{\mathbb F}(\cdot)\phi\Bigl(\bigl\| F(\cdot +{\bf b}_{k_{m}};x)-F_{N}(\cdot +{\bf b}_{k_{m}};x)\bigr\|_{Y}\Bigr) \Biggr\|_{P}
\\& +\sup_{x\in B}\Biggl\|{\mathbb F}(\cdot)\phi\Bigl(\bigl\| F_{N}(\cdot +{\bf b}_{k_{m}};x)-F_{N}(\cdot +{\bf b}_{k_{m'}};x)\bigr\|_{Y}\Bigr) \Biggr\|_{P}
\\& +\sup_{x\in B}\Biggl\|{\mathbb F}(\cdot)\phi\Bigl(\bigl\| F_{N}(\cdot +{\bf b}_{k_{m'}};x)-F(\cdot +{\bf b}_{k_{m'}};x)\bigr\|_{Y}\Bigr) \Biggr\|_{P}.
\end{align*}
\end{proof}

\begin{rem}\label{smjeh}
It is clear that condition (ii) from the formulations of Proposition \ref{kimih} and Proposition \ref{asad} is a little bit inappropriate because it cannot detect a continuity or a measurability of function $w(\cdot).$ Despite of this, the proofs of  Proposition \ref{kimih}, Proposition \ref{asad} and many other statements clarified below still work in concrete situations, so that our results are applicable if $P$ is a solid Banach space, for example.\index{solid Banach space}
\end{rem}

We need the following extension of \cite[Definition 3.1]{metrical}, where we have ${\mathbb F}(\cdot)\equiv 1$ and $\phi(x)\equiv x:$

\begin{defn}\label{nafaks-met}
Suppose that $\emptyset  \neq \Lambda' \subseteq {\mathbb R}^{n},$ $\emptyset  \neq \Lambda \subseteq {\mathbb R}^{n},$ $F : \Lambda \times X \rightarrow Y$ is a given function, $\rho$ is a binary relation on $Y,$ ${\mathbb F} : \Lambda \rightarrow (0,\infty)$ and $\Lambda' \subseteq \Lambda''.$ Then we say that:
\begin{itemize}
\item[(i)]\index{function!Bohr $(\phi,{\mathbb F},{\mathcal B},\Lambda',\rho,{\mathcal P})$-almost periodic}
$F(\cdot;\cdot)$ is Bohr $(\phi,{\mathbb F},{\mathcal B},\Lambda',\rho,{\mathcal P})$-almost periodic if and only if for every $B\in {\mathcal B}$ and $\epsilon>0$
there exists $l>0$ such that for each ${\bf t}_{0} \in \Lambda'$ there exists ${\bf \tau} \in B({\bf t}_{0},l) \cap \Lambda'$ such that, for every ${\bf t}\in \Lambda$ and $x\in B,$ there exists an element $y_{{\bf t};x}\in \rho (F({\bf t};x))$ such that
\begin{align*}
\sup_{x\in B} \Biggl\| {\mathbb F}(\cdot) \phi\Bigl(\bigl\| F(\cdot+{\bf \tau};x)-y_{\cdot;x}\bigr\|_{Y}\Bigr)\Biggr\|_{P} \leq \epsilon .
\end{align*}
\item[(ii)] \index{function!$(\phi,{\mathbb F},{\mathcal B},\Lambda',\rho,{\mathcal P})$-uniformly recurrent}
$F(\cdot;\cdot)$ is $(\phi,{\mathbb F},{\mathcal B},\Lambda',\rho,{\mathcal P})$-uniformly recurrent if and only if for every $B\in {\mathcal B}$ 
there exists a sequence $({\bf \tau}_{k})$ in $\Lambda'$ such that $\lim_{k\rightarrow +\infty} |{\bf \tau}_{k}|=+\infty$ and that, for every ${\bf t}\in \Lambda$ and $x\in B,$ there exists an element $y_{{\bf t};x}\in \rho (F({\bf t};x))$ such that 
\begin{align*}
\lim_{k\rightarrow +\infty}\sup_{x\in B} \Biggl\| {\mathbb F}(\cdot)\phi\Bigl(\bigl\|F(\cdot+{\bf \tau}_{k};x)-y_{\cdot;x}\bigr\|_{Y}\Bigr)\Biggr\|_{P}=0.
\end{align*}
\end{itemize}
\end{defn}

For our purposes, the situation in which $\rho={\mathrm I}$ will be the most important (the consideration of notion of strong $(\phi,\rho,{\mathbb F},{\mathcal B},{\mathcal P})$-almost periodicity with a general binary relation $\rho$ on $Y$ is a bit misleading; cf. Definition \ref{strong-app}). 

We continue by clarifying the following result:

\begin{prop}\label{rep}
\begin{itemize}
\item[(i)] Suppose that ${\mathrm R}$ is any collection of sequences in $\Lambda'',$ $F: \Lambda \times X \rightarrow Y,$ $\phi : [0,\infty) \rightarrow [0,\infty)$ and ${\mathbb F} : \Lambda \rightarrow (0,\infty).$ Let the following conditions hold:
\begin{itemize}
\item[(a)] The function $\phi(\cdot)$ is monotonically increasing, continuous at the point zero, and there exists a finite real constant $c>0$ such that $\phi(x+y)\leq c[\phi(x)+\phi(y)]$ for all $x,\ y\geq 0.$
\item[(b)] Condition \emph{(C3)} holds.
\item[(c)] ${\mathbb F}(\cdot)\phi( \| P(\cdot;x) \|_{Y})\in P$ for any trigonometric polynomial (periodic function) $P(\cdot;\cdot)$ and $x\in X.$
\item[(d)]  There exists a finite real constant $g>0$ such that 
$$
\Bigl\|{\mathbb F}(\cdot)\phi \bigl( \| P(\cdot;x) \|_{Y}\bigr)\Bigr\|_{P} \leq g \Bigl\| \phi \bigl( \| P(\cdot;x) \|_{Y}\bigr)\Bigr\|_{\infty},
$$ 
for any any trigonometric polynomial (periodic function) $P(\cdot;\cdot)$ and $x\in X.$
\item[(e)] There exists
a finite real constant $h>0$ such that, for every $x\in X$ and $\tau \in \Lambda'',$ the assumption ${\mathbb F}(\cdot)\phi(\| H(\cdot ; x)\|_{Y}) \in P$ implies
${\mathbb F}(\cdot)\phi(\| H(\cdot+\tau ;x)\|_{Y}) \in P$ and 
$$
\Bigl\| {\mathbb F}(\cdot)\phi \bigl(\| H(\cdot+\tau ;x)\|_{Y}\bigr) \Bigr\|_{P}\leq h\Bigl\|{\mathbb F}(\cdot)\phi\bigl(\| H(\cdot ;x)\|_{Y}\bigr) \Bigr\|_{P}.
$$
\item[(f)] Any set $B$ of collection ${\mathcal B}$ is bounded.
\end{itemize}
If $F\in e-({\mathcal B},\phi,{\mathbb F})-B^{{\mathcal P}_{\cdot}}(\Lambda \times X : Y)$ [$e-({\mathcal B},\phi,{\mathbb F})_{j\in {\mathbb N}_{n}}-B^{{\mathcal P}_{\cdot}}_{{\rm I}}(\Lambda \times X : Y)$], then  the function $F(\cdot;\cdot)$ is
$(\phi,{\mathrm R}, {\mathcal B},\phi,{\mathbb F},{\mathcal P})$-normal.
\item[(ii)] Suppose that $\emptyset  \neq \Lambda' \subseteq {\mathbb R}^{n},$ $\emptyset  \neq \Lambda \subseteq {\mathbb R}^{n},$ $F : \Lambda \times X \rightarrow Y$ is a given function, and $\Lambda' \subseteq \Lambda''.$ If $F\in e-({\mathcal B},\phi,{\mathbb F})-B^{{\mathcal P}_{\cdot}}(\Lambda \times X : Y)$ [$e-({\mathcal B},\phi,{\mathbb F})_{j\in {\mathbb N}_{n}}-B^{{\mathcal P}_{\cdot}}_{{\rm I}}(\Lambda \times X : Y)$] and the assumptions \emph{(a)-(f)} given in the formulation of \emph{(i)} hold, then the function $F(\cdot;\cdot)$ is Bohr $(\phi,{\mathbb F},{\mathcal B},\Lambda',{\rm I},{\mathcal P})$-almost periodic.
\end{itemize}
\end{prop}

\begin{proof}
We will prove (i) for the function $F\in e-({\mathcal B},\phi,{\mathbb F})-B^{{\mathcal P}_{\cdot}}(\Lambda \times X : Y)$. Let a real number $\epsilon>0,$ a set $B\in {\mathcal B},$ and a sequence $({\bf b}_{k})_{k\in {\mathbb N}}$ in ${\mathrm R}$ be fixed. Then we can find a trigonometric polynomial $P(\cdot;\cdot)$ such that
$$
\sup_{x\in B}\Bigl\|{\mathbb F}(\cdot) \phi\Bigl(\bigl\|P(\cdot;x)-F(\cdot;x)\bigr\|_{Y}\Bigr)\Bigr\|_{P}<\epsilon/3.
$$
After that we can use condition (e) in order to see that 
\begin{align}\label{crisis}
\sup_{x\in B}\Bigl\|{\mathbb F}(\cdot) \phi\Bigl(\bigl\|P(\cdot+\tau ;x)-F(\cdot+\tau;x)\bigr\|_{Y}\Bigr)\Bigr\|_{P}<h\epsilon/3
\end{align}
for all $\tau \in \Lambda''.$ Since $B$ is bounded, the Bochner criterion for almost periodic functions in ${\mathbb R}^{n}$ ensures that there exist a subsequence $({\bf b}_{k_{m}})_{m\in {\mathbb N}}$ of $({\bf b}_{k})_{k\in {\mathbb N}}$ and a natural number $m_{0}\in {\mathbb N}$ such that, for every positive integers $m',\ m''\geq m_{0},$ we have
$$
\sup_{x\in B}\Bigl\| P\bigl({\bf t}+{\bf b}_{k_{m'}};x\bigr)-P\bigl({\bf t}+{\bf b}_{k_{m''}};x\bigr)\Bigr\|_{Y}<\epsilon/3,\quad {\bf t}\in {\mathbb R}^{n}.
$$
Employing (c)-(d), we get:
$$
\sup_{x\in B}\Biggl\| {\mathbb F}(\cdot)\phi\Bigl( \Bigl\| P({\bf t}+{\bf b}_{k_{m'}};x)-P({\bf t}+{\bf b}_{k_{m''}};x)\Bigr\|_{Y}\Bigr) \Biggr\|_{P}\leq  g \phi(\epsilon/3),\quad m',\ m''\geq m_{0}.
$$
Then the final conclusion simply follows from conditions (a)-(b), the estimate \eqref{crisis} and the next decomposition
\begin{align*}
\Biggl\|{\mathbb F}(\cdot)&\phi\Bigl( \bigl\| F(\cdot+{\bf b}_{k_{m'}};x)-F(\cdot+{\bf b}_{k_{m''}};x)\bigr\|_{Y} \Bigr)\Biggr\|_{P} 
\\& \leq c'\Biggl[\Biggl\|{\mathbb F}(\cdot)\phi\Bigl( \bigl\| F(\cdot+{\bf b}_{k_{m'}};x)-P(\cdot+{\bf b}_{k_{m'}};x)\bigr\|_{Y} \Bigr)\Biggr\|_{P} 
\\&+\Biggl\|{\mathbb F}(\cdot)\phi\Bigl( \bigl\| P(\cdot+{\bf b}_{k_{m'}};x)-P(\cdot+{\bf b}_{k_{m''}};x)\bigr\|_{Y} \Bigr)\Biggr\|_{P} 
\\&+\Biggl\|{\mathbb F}(\cdot)\phi\Bigl( \bigl\| P(\cdot+{\bf b}_{k_{m''}};x)-F(\cdot+{\bf b}_{k_{m''}};x)\bigr\|_{Y} \Bigr)\Biggr\|_{P} \Biggr],
\end{align*}
with a certain positive real constant $c'>0.$
The proof of (ii) can be deduced similarly and therefore omitted.
\end{proof}

The conclusions established in \cite[Example 3.9]{metrical} can be simply formulated for Bohr $(\phi,{\mathbb F},{\mathcal B},\Lambda',\rho,{\mathcal P})$-almost periodic functions and $(\phi,{\mathbb F},{\mathcal B},\Lambda',\rho,{\mathcal P})$-uniformly recurrent functions. Furthermore, the statements of \cite[Proposition 3.7, Corollary 3.8]{metrical} can be extended in the following way:

\begin{prop}\label{nbm}
Suppose that $\emptyset  \neq \Lambda' \subseteq {\mathbb R}^{n},$ $\emptyset  \neq \Lambda \subseteq {\mathbb R}^{n},$  $\Lambda +\Lambda' \subseteq \Lambda$, and the function $F : \Lambda \times X \rightarrow Y$ is Bohr $(\phi,{\mathbb F},{\mathcal B},\Lambda',\rho,{\mathcal P})$-almost periodic ($(\phi,{\mathbb F},{\mathcal B},\Lambda',\rho,{\mathcal P})$-uniformly recurrent), where $\rho$ is a binary relation on $Y$ satisfying $R(F)\subseteq D(\rho)$ and $\rho(y)$ is a singleton for any $y\in R(F).$ Suppose that for each ${\bf \tau}\in \Lambda'$ we have $\tau +\Lambda=\Lambda$, the function $\phi(\cdot)$ is monotonically increasing, there exists a finite real constant $c>0$ such that $\phi(x+y)\leq c[\phi(x)+\phi(y)]$ for all $x,\ y\geq 0,$  
and the following conditions hold:
\begin{itemize}\label{ijk}
\item[(P)] ${\mathcal P}_{1}=(P_{1},d_{1})$ is a pseudometric space, $c'\in (0,\infty)$ and for every $f\in P$ and $\tau \in \Lambda'$ we have $f(\cdot-\tau)\in P_{1}$ and $\| f(\cdot-\tau)\|_{P_{1}}\leq c'\|f\|_{P}.$ 
\item[(i)] We have ${\mathbb G} : \Lambda \rightarrow (0,\infty)$ and ${\mathbb G}(\cdot) \leq \inf_{\tau \in \Lambda'}{\mathbb F}(\cdot-\tau).$
\item[(ii)] Condition \emph{(C2)} holds.
\item[(iii)] Condition \emph{(C0)} holds, where:
\begin{itemize}
\item[(C0)]
There exists a finite real constant $e>0$ such that the assumptions $g\in P_{1}$ and $0\leq f\leq g$ imply $\|f\|_{P_{1}}\leq e\|g\|_{P_{1}}.$
\end{itemize}
\end{itemize}
Then $\Lambda+(\Lambda'-\Lambda')\subseteq \Lambda$ and the function $F(\cdot;\cdot)$ is Bohr $(\phi,{\mathbb F},{\mathcal B},\Lambda'-\Lambda',{\rm I},{\mathcal P})$-almost periodic ($(\phi,{\mathbb F},{\mathcal B},\Lambda'-\Lambda',{\rm I},{\mathcal P})$-uniformly recurrent).
\end{prop}\index{condition!(C0)}

\begin{proof}
We will consider only Bohr $(\phi,{\mathbb F},{\mathcal B},\Lambda'-\Lambda',{\rm I},{\mathcal P})$-almost periodic functions. The inclusion $\Lambda+(\Lambda'-\Lambda')\subseteq \Lambda$ can be simply verified.
Let a real number $\epsilon>0$ and a set $B\in {\mathcal B}$ be given. Then there exists $l>0$ such that for each ${\bf t}_{0}^{1},\  {\bf t}_{0}^{2} \in \Lambda'$ there exist two points ${\bf \tau}_{1} \in B({\bf t}_{0}^{1},l) \cap \Lambda'$ and
${\bf \tau}_{2} \in B({\bf t}_{0}^{2},l) \cap \Lambda'$
such that, for every $x\in B,$ we have
\begin{align*}
\Biggl\|{\mathbb F}(\cdot) \phi\Bigl(\bigl\|F\bigl(\cdot+{\bf \tau}_{1};x\bigr)-\rho(F(\cdot;x))\bigr\|_{Y}\Bigr)\Biggr\|_{P} \leq \epsilon/2
\end{align*}
and
\begin{align*}
\Biggl\|{\mathbb F}(\cdot) \phi\Bigl(\bigl\| F\bigl(\cdot+{\bf \tau}_{2};x\bigr)-\rho(F(\cdot;x))\bigr\|_{Y}\Bigr)\Biggr\|_{P} \leq \epsilon/2.
\end{align*}
Our assumptions on the function $\phi(\cdot)$ and condition (ii) simply imply:
\begin{align*}
\Biggl\| {\mathbb F}(\cdot) \phi\Bigl(\bigl\| F\bigl(\cdot+{\bf \tau}_{1};x\bigr)-F\bigl(\cdot+{\bf \tau}_{2};x\bigr)\bigr\|_{Y}\Bigr)\Biggr\|_{P} \leq c'\epsilon,\ x\in B,
\end{align*}
with a certain positive real constant $c'>0.$
Using (P) and translation for the vector $-\tau_{2},$ we get 
\begin{align*}
\Biggl\|{\mathbb F}\bigl(\cdot -\tau_{2}\bigr) \phi\Bigl(\bigl\|F\bigl(\cdot+\bigl[\tau_{2}-{\bf \tau}_{1}\bigr];x\bigr)-F\bigl(\cdot;x\bigr)\bigr\|_{Y}\Bigr)\Biggr\|_{P_{1}} \leq c''\epsilon,\ x\in B,
\end{align*}
with a certain positive real constant $c''>0.$
Clearly, $\tau_{2}-\tau_{1}\in B({\bf t}_{0}^{2}-{\bf t}_{0}^{1},2l) \cap (\Lambda'-\Lambda')$; keeping in mind the assumptions (i) and (iii), the last inclusion simply implies the required.
\end{proof}

\subsection{Metrical semi-$(c_{j},{\mathcal B})_{j\in {\mathbb N}_{n}}$-periodic functions}\label{metr-21}

Let us recall that
the notion of semi-periodicity (sometimes also called limit-periodicity) has been thoroughly analyzed by
J. Andres and D. Pennequin in \cite{andres}. 
The main aim of this subsection is to continue our recent analysis of 
semi-$(c_{j},{\mathcal B})_{j\in {\mathbb N}_{n}}$-periodic functions carried out in \cite[Subsection 2.1]{brno}; for simplicity, we will not consider general binary relations here, and we will assume that $c_{j}\in \{z\in {\mathbb C} : |z|=1\}$ for all $j\in {\mathbb N}_{n}$.

We start by introducing the following metrical analogue of \cite[Definition 2.9]{brno}:

\begin{defn}\label{brno-semi}
Suppose that $\emptyset  \neq \Lambda \subseteq {\mathbb R}^{n}$ and $F :\Lambda  \times X \rightarrow Y.$
Then we say that $F(\cdot;\cdot)$ is semi-$(\phi,c_{j},{\mathbb F},{\mathcal B},{\mathcal P})_{j\in {\mathbb N}_{n}}$-periodic of type $1$ if and only if for each $B\in {\mathcal B}$ and $\epsilon>0$ there exist non-zero real numbers $\omega_{j}$ such that $\omega_{j}e_{j}\in \Lambda''$ for all $j\in {\mathbb N}_{n}$ and
\begin{align}\label{gnarbis}
\sup_{x\in B} \Bigl\|{\mathbb F}(\cdot) \phi\Bigl(\bigl\|F(\cdot +m\omega_{j}e_{j};x)-c_{j}^{m}F(\cdot;x)\bigr\|_{Y}\Bigr)\Bigr\|_{P}<\epsilon,\quad m\in {\mathbb N},\ j\in {\mathbb N}_{n}.
\end{align}
\end{defn}\index{function!semi-$(\phi,c_{j},{\mathbb F},{\mathcal B},{\mathcal P})_{j\in {\mathbb N}_{n}}$-periodic of type $1$}

In the metrical framework, the use of \eqref{gnarbis} with $m\in {\mathbb N}$
is not satisfactorily enough if we use the (pseudo-)metric spaces ${\mathcal P}$ different from $l^{\infty}(\Lambda).$ The following definition suggests the use of \eqref{gnarbis} with $m\in {\mathbb Z}$ and the region $\Lambda={\mathbb R}^{n}:$

\begin{defn}\label{brno-semi2}
Suppose that $F : {\mathbb R}^{n}  \times X \rightarrow Y.$
Then we say that $F(\cdot;\cdot)$ is semi-$(\phi,c_{j},{\mathbb F},{\mathcal B},{\mathcal P})_{j\in {\mathbb N}_{n}}$-periodic of type $2$ if and only if for each $B\in {\mathcal B}$ and $\epsilon>0$ there exist non-zero real numbers $\omega_{j}$ such that \eqref{gnarbis} holds for all $m\in {\mathbb Z}$ and $j\in {\mathbb N}_{n}.$
\end{defn}\index{function!semi-$(\phi,c_{j},{\mathbb F},{\mathcal B},{\mathcal P})_{j\in {\mathbb N}_{n}}$-periodic of type $2$}

In what follows, we investigate the relationship between the notions of semi-$(\phi,c_{j},{\mathbb F},{\mathcal B},{\mathcal P})_{j\in {\mathbb N}_{n}}$-periodicity and the semi-$(\phi,c_{j},{\mathbb F},{\mathcal B},{\mathcal P})_{j\in {\mathbb N}_{n}}$-periodicity of type $1$ ($2$); cf. also \cite[Theorem 2.10]{brno}:

\begin{prop}\label{naiv12}
Suppose that condition \emph{(C3)} holds, $\phi(0)=0$ and there exists a finite real constant $c>0$ such that $\phi(x+y)\leq c[\phi(x)+\phi(y)]$ for all $x,\ y\geq 0.$ If, for every $x\in X$ and $ \tau \in \Lambda'',$ the assumption ${\mathbb F}(\cdot) \phi(\|G(\cdot ;x)\|_{Y})\in P$ implies ${\mathbb F}(\cdot) \phi(\|G(\cdot +\tau;x)\|_{Y})\in P$ and
\begin{align}\label{abso}
\Bigl\|{\mathbb F}(\cdot) \phi \bigl(\|G(\cdot +\tau;x)\|_{Y}\bigr) \Bigr\|_{P} \leq \Bigl\|{\mathbb F}(\cdot) \phi \bigl(\|G(\cdot ;x)\|_{Y}\bigr) \Bigr\|_{P},
\end{align}
and the function $F :\Lambda  \times X \rightarrow Y$
is semi-$(\phi,c_{j},{\mathbb F},{\mathcal B},{\mathcal P})_{j\in {\mathbb N}_{n}}$-periodic ($\Lambda={\mathbb R}^{n}$), then $F(\cdot;\cdot)$ is semi-$(\phi,c_{j},{\mathbb F},{\mathcal B},{\mathcal P})_{j\in {\mathbb N}_{n}}$-periodic of type $1$ (type $2$).
\end{prop}

\begin{proof}
We will consider the semi-$(\phi,c_{j},{\mathbb F},{\mathcal B},{\mathcal P})_{j\in {\mathbb N}_{n}}$-periodic functions of type $1$, only.
Let $\epsilon>0$ and $B\in {\mathcal B}$ be fixed. Then there exists a $(c_{j})_{j\in {\mathbb N}_{n}}$-periodic function $P(\cdot;\cdot)$ and non-zero real numbers $\omega_{j}$ such that $\omega_{j}e_{j}\in \Lambda''$ for all $j\in {\mathbb N}_{n}$ and
\begin{align}\label{gnarbis1}
\sup_{x\in B} \Bigl\|{\mathbb F}(\cdot) \phi\Bigl(\bigl\|P(\cdot ;x)-F(\cdot;x)\bigr\|_{Y}\Bigr)\Bigr\|_{P}<\epsilon/2.
\end{align}
Using the prescribed assumptions, we have the existence of a finite real constant $c'>0$ such that:
\begin{align*}
\sup_{x\in B}& \Bigl\|{\mathbb F}(\cdot) \phi\Bigl(\bigl\|F(\cdot +m\omega_{j}e_{j};x)-c_{j}^{m}F(\cdot;x)\bigr\|_{Y}\Bigr)\Bigr\|_{P}
\\& \leq c'\Biggl[\sup_{x\in B} \Bigl\|{\mathbb F}(\cdot) \phi\Bigl(\bigl\|F(\cdot +m\omega_{j}e_{j};x)-P(\cdot +m\omega_{j}e_{j};x)\bigr\|_{Y}\Bigr)\Bigr\|_{P}
\\&+\sup_{x\in B} \Bigl\|{\mathbb F}(\cdot) \phi\Bigl(\bigl\|P(\cdot +m\omega_{j}e_{j};x)-c_{j}^{m}P(\cdot;x)\bigr\|_{Y}\Bigr)\Bigr\|_{P}
\\& +\sup_{x\in B} \Bigl\|{\mathbb F}(\cdot) \phi\Bigl(\bigl\|c_{j}^{m}P(\cdot;x)-c_{j}^{m}F(\cdot;x)\bigr\|_{Y}\Bigr)\Bigr\|_{P}
\Biggr]
\\& =c'\Biggl[\sup_{x\in B} \Bigl\|{\mathbb F}(\cdot) \phi\Bigl(\bigl\|F(\cdot +m\omega_{j}e_{j};x)-P(\cdot +m\omega_{j}e_{j};x)\bigr\|_{Y}\Bigr)\Bigr\|_{P}
\\& +\sup_{x\in B} \Bigl\|{\mathbb F}(\cdot) \phi\Bigl(\bigl\|P(\cdot;x)-F(\cdot;x)\bigr\|_{Y}\Bigr)\Bigr\|_{P}
\Biggr]
,\quad m\in {\mathbb N},\ j\in {\mathbb N}_{n}.    
\end{align*}
Then the estimate \eqref{gnarbis} simply follows using the assumption \eqref{abso} and the estimate \eqref{gnarbis1}.
\end{proof}

\begin{rem}\label{raqlaq}
It is worth noting that, in the usually considered situation when $\phi(x)\equiv x$ and ${\mathbb F}(\cdot)\equiv 1,$ the assumptions of Proposition \ref{naiv12} are satisfied for the $p$-semi-$c$-periodic functions
$F : {\mathbb R}\rightarrow Y$ in variation ($1\leq p<+\infty$) since the metric $d(\cdot;\cdot)$ introduced in \eqref{12589} is translation invariant, i.e., for every $\tau \in {\mathbb R}$ we have $d(f,g)=d(f(\cdot +\tau),g(\cdot +\tau)),$ with the meaning clear.
\end{rem}

The converse of Proposition \ref{naiv12} cannot be so simply formulate in the metrical framework because the authors of \cite{andres} has used, in the proof of \cite[Lemma 1]{andres}, the linearizations of periodic functions which approximates $F(\cdot).$ Concerning this issue, we will only formulate one research result without proof, because it is very similar to the proof of \cite[Theorem 1]{andres}:

\begin{prop}\label{naiv122}
Suppose that $P=C_{b,\nu}({\mathbb R}^{n} : [0,\infty))$, the function $F : {\mathbb R}^{n} \rightarrow Y$ is continuous and semi-$(\phi,c_{j},1,{\mathcal P})_{j\in {\mathbb N}_{n}}$-periodic of type $2$, where the function $\nu(\cdot)$ is bounded from above, the function $\phi(\cdot)$ is monotonically increasing, and there exists a finite real constant $c>0$ such that $\phi(x+y)\leq c[\phi(x)+\phi(y)]$ for all $x,\ y\geq 0.$
Then the function $F (\cdot)$ is semi-$(\phi,c_{j},1,{\mathcal P})_{j\in {\mathbb N}_{n}}$-periodic.
\end{prop}

\section{Metrical approximations: Stepanov, Weyl, Besicovitch and Doss concepts}\label{general-stepa}

In this section, we will consider the generalized Stepanov, Weyl, Besicovitch and Doss metrical approximations by trigonometric polynomials and $\rho$-almost periodic type functions.

\subsection{Stepanov and Weyl metrical approximations}\label{anadol}
In this part, we will assume that condition (S) holds, where:\index{condition!(S)}
\begin{itemize}
\item[(S)] Let $\Omega$ be any compact subset of ${\mathbb R}^{n}$ with positive Lebesgue measure such that $\Lambda +\Omega \subseteq \Lambda.$ 
We assume that
$P_{\Omega} \subseteq [0,\infty)^{\Omega},$  the zero function belongs to $P_{\Omega}$, and 
${\mathcal P}_{\Omega}=(P_{\Omega},d_{\Omega})$ is a pseudometric space. Let $P \subseteq [0,\infty)^{\Lambda}$, let the zero function belong to $P$, and let 
${\mathcal P}=(P,d)$ be a pseudometric space. 
\end{itemize}

By $\cdot$ and $\cdot \cdot$ we denote the arguments from $\Omega$ and $\Lambda,$ respectively. We will use a similar notation in the analysis of Weyl metrical approximations.

We will first 
introduce the following notion:

\begin{defn}\label{stepanov-approximation}
Suppose that (S) holds, $\phi : [0,\infty) \rightarrow [0,\infty)$ and ${\mathbb F}: \Omega \times \Lambda \rightarrow (0,\infty).$ 
Then we say that the function $F(\cdot;\cdot)$ belongs to the class
$e-({\mathbb F},{\mathcal B})-S^{{\mathcal P}_{\Omega}}(\Lambda \times X : Y)$ [$e-({\mathbb F},{\mathcal B})-S^{{\mathcal P}_{\Omega}}_{\rho}(\Lambda \times X : Y);$ $e-({\mathbb F},{\mathcal B})_{j\in {\mathbb N}_{n}}-S^{{\mathcal P}_{\Omega}}_{\rho_{j}}(\Lambda \times X : Y)$] if and only if
for every $B\in {\mathcal B}$ and for every $\epsilon>0$ there exists a trigonometric polynomial [$\rho$-periodic function; $(\rho_{j})_{j\in {\mathbb N}_{n}}$-periodic function] $P(\cdot;\cdot)$ such that
\begin{align*}
\sup_{x\in B}\Biggl\| \Bigl\| {\mathbb F}(\cdot,\cdot \cdot ) \phi\Bigl( \bigl\|P(\cdot \cdot +\cdot ;x)-F(\cdot \cdot+\cdot ;x)\bigr\|_{Y}\Bigr) \Bigr\|_{P_{\Omega}}\Biggr\|_{P}<\epsilon .
\end{align*}
\end{defn}\index{space!$e-({\mathbb F},{\mathcal B})-S^{{\mathcal P}_{\Omega}}(\Lambda \times X : Y)$}\index{space!$e-({\mathbb F},{\mathcal B})-S^{{\mathcal P}_{\Omega}}_{\rho}(\Lambda \times X : Y)$}\index{space!$e-({\mathbb F},{\mathcal B})_{j\in {\mathbb N}_{n}}-S^{{\mathcal P}_{\Omega}}_{\rho_{j}}(\Lambda \times X : Y)$}

The most important subclass of the general class introduced above is obtained by plugging
${\mathbb F}(\cdot,\cdot \cdot ) \equiv 1,$ 
$P=L^{\infty}(\Lambda)$ and $P_{\Omega}=L^{p}(\Omega),$ where $1\leq p<\infty.$

Besides the notion introduced in Definition \ref{stepanov-approximation}, we will consider the following ones:

\begin{defn}\label{stepanov-approximation1}
Suppose that (S) holds, $\phi : [0,\infty) \rightarrow [0,\infty),$ ${\mathbb F}: \Omega \times \Lambda \rightarrow (0,\infty),$ and $F: \Lambda \times X \rightarrow Y.$ 
\begin{itemize}
\item[(i)]\index{function!Stepanov
$(\phi,{\mathrm R}, {\mathcal B},\phi,{\mathbb F},{\mathcal P})$-normal}
Suppose that ${\mathrm R}$ is any collection of sequences in $\Lambda''.$ Then we say that the function $F(\cdot;\cdot)$ is Stepanov
$(\phi,{\mathrm R}, {\mathcal B},\phi,{\mathbb F},{\mathcal P})$-normal
if and only if for every set $B\in {\mathcal B}$ and for every sequence $({\bf b}_{k})_{k\in {\mathbb N}}$ in ${\mathrm R}$ there exists a subsequence $({\bf b}_{k_{m}})_{m\in {\mathbb N}}$ of $({\bf b}_{k})_{k\in {\mathbb N}}$ such that, for every $\epsilon>0,$ there exists an integer $m_{0}\in {\mathbb N}$ such that, for every integers $m,\ m'\geq m_{0},$ we have 
\begin{align*}
\sup_{x\in B}\Biggl\| \Bigl\| {\mathbb F}(\cdot,\cdot \cdot ) \phi\Bigl( \bigl\| F(\cdot+\cdot \cdot+{\bf b}_{k_{m}};x)-F(\cdot+\cdot \cdot+{\bf b}_{k_{m'}};x)\bigr\|_{Y}\Bigr) \Bigr\|_{P_{\Omega}}\Biggr\|_{P}<\epsilon .
\end{align*}\index{space!$S^{(\phi,{\mathbb F},\rho,{\mathcal P}_{\Omega},{\mathcal P})}_{\Omega,\Lambda',{\mathcal B}}(\Lambda\times X :Y)$}
\item[(ii)] By $S^{(\phi,{\mathbb F},\rho,{\mathcal P}_{\Omega},{\mathcal P})}_{\Omega,\Lambda',{\mathcal B}}(\Lambda\times X :Y)$ we denote the set consisting of all functions $F : \Lambda \times X \rightarrow Y$ such that, for every $\epsilon>0$ and $B\in {\mathcal B},$ there exists a finite real number
$L>0$ such that for each ${\bf t}_{0}\in \Lambda'$ there exists $\tau \in B({\bf t}_{0},L)\cap \Lambda'$ such that, for every $x\in B,$ the mapping ${\bf u}\mapsto G_{x}({\bf u})\in \rho( F({\bf u};x)),$ ${\bf u}\in \Omega +\Lambda$ is well defined, and
\begin{align*}
\sup_{x\in B}\Biggl\|{\Bigl\| \mathbb F}(\cdot;\cdot \cdot)\phi\Bigl( \bigl\| F({\bf \tau}+\cdot+\cdot \cdot;x)-G_{x}(\cdot+\cdot \cdot)\bigr\|_{Y}\Bigr) \Bigr\|_{P_{\Omega}} \Biggr\|_{P}<\epsilon.
\end{align*}
\end{itemize}
\end{defn}

The notion introduced in Definition \ref{stepanov-approximation} is stronger than the notion introduced in Definition \ref{stepanov-approximation1};
the interested reader may simply clarify some sufficient conditions under which a function $F\in e-({\mathbb F},{\mathcal B})-S^{{\mathcal P}_{\Omega}}(\Lambda \times X : Y)$ [$F\in e-({\mathbb F},{\mathcal B})_{j\in {\mathbb N}_{n}}-S^{{\mathcal P}_{\Omega}}_{{\rm I}}(\Lambda \times X : Y)$] is 
Stepanov
$(\phi,{\mathrm R}, {\mathcal B},\phi,{\mathbb F},{\mathcal P}_{\Omega},{\mathcal P})$-normal or belongs to the class 
$S^{(\phi,{\mathbb F},\rho,{\mathcal P}_{\Omega},{\mathcal P})}_{\Omega,\Lambda',{\mathcal B}}(\Lambda\times X :Y).$ In the usual context, we have already discussed this question in Proposition \ref{rep} (a uniformly recurrent analogue of the notion introduced in Definition \ref{stepanov-approximation1}(ii) can be also analyzed).

In Definition \ref{stepanov-approximation1}(ii), we do not explicitly use the notion of multi-dimensional Bochner transform (\cite{nova-selected}). In the usual setting, this approach is 
commonly used and it is a very special case of the general approach obeyed for the introduction of function spaces in Definition \ref{stepanov-approximation} and Definition \ref{stepanov-approximation1}.

We have the following:

\begin{example}\label{kakad321}
Suppose that $1\leq p<\infty,$ $c\in \{  z\in {\mathbb C} : |z|=1\}$ and $f\in L_{loc}^{p}({\mathbb R} : Y).$ The one-dimensional Bochner transform $\hat{f} : {\mathbb R} \rightarrow L^{p}([0,1] : Y)$ is defined by $[\hat{f}(t)](s):=f(t+s),$ $t\in {\mathbb R},$ $s\in [0,1];$ a function $f(\cdot)$ is said to be Stepanov-$(p,c)$-semi-periodic if and only if the function $\hat{f}(\cdot)$ is semi-$c$-periodic (\cite{nova-mono}). 

Suppose now that $Y={\mathbb R},$ the function $f(\cdot)$ is semi-$c$-periodic and can be analytically extended to a strip around the real axis. Then the function sign$(f(\cdot))$
is Stepanov-$(p,c)$-semi-periodic. In actual fact, there exists a sequence $(f_{k})$ of $c$-periodic functions which converges uniformly to the function $f(\cdot)$ on the real line. Then $\hat{f_{k}}: {\mathbb R} \rightarrow L^{p}([0,1] : {\mathbb R})$ is $c$-periodic and the required conclusion simply follows if we prove that, for every $\epsilon>0$, there exists an integer $k_{0}\in {\mathbb N}$ such that, for every $k\geq k_{0},$ we have
\begin{align}\label{smuk}
\int^{t+1}_{t}\bigl| \mbox{sign}\bigl( f_{k}(s)\bigr)- \mbox{sign}\bigl( f(s)\bigr) \bigr|^{p}\, ds \leq \epsilon,\quad t\in {\mathbb R}.
\end{align}
By the proof of \cite[Theorem 5.3.1]{188}, we have that 
there exists a sufficiently small real number $\epsilon_{0}>0$ such that $m(\{ x\in [t,t+1] : |f(x)|\leq  \epsilon_{0}\})\leq 2^{-p}\epsilon$ for all $t\in {\mathbb R}.$ Let $k_{0}\in {\mathbb N}$ be such that $|f_{k}(t)-f(t)|\leq \epsilon_{0}/2$ for all $k\geq k_{0}$ and $t\in {\mathbb R}.$ If $|f(s)|\geq \epsilon_{0},$ then we have $|f_{k}(s)|\geq \epsilon_{0}/2$ and $\mbox{sign}( f_{k}(s))= \mbox{sign}( f(s)) $ for all $k\geq k_{0}$ ($s\in {\mathbb R}$). Therefore,
\begin{align*}
\int^{t+1}_{t}&\bigl| \mbox{sign}\bigl( f_{k}(s)\bigr)- \mbox{sign}\bigl( f(s)\bigr) \bigr|^{p}\, ds 
\\&=\int_{\{ x\in [t,t+1] : |f(x)|\leq  \epsilon_{0}\}}\bigl| \mbox{sign}\bigl( f_{k}(s)\bigr)- \mbox{sign}\bigl( f(s)\bigr) \bigr|^{p}\, ds 
\\& \leq 2^{p}m\bigl(\bigl\{ x\in [t,t+1] : |f(x)|\leq  \epsilon_{0}\bigr\}\bigr)\leq \epsilon,
\end{align*}
so that \eqref{smuk} holds.
\end{example}

We continue our exposition with the following example:

\begin{example}\label{kakad}
Suppose that $1\leq p<\infty.$
It is well known that the function 
$$
f(t):=\sin\Biggl( \frac{1}{2+\cos t +\cos (\sqrt{2}t)}\Biggr),\quad t\in {\mathbb R}
$$\index{function!Lipschitz $S^{p}$-almost periodic}
is Stepanov-$p$-almost periodic ($S^{p}$-almost periodic); see \cite{nova-mono}. Now we will prove that the function $f(\cdot)$ is not Lipschitz $S^{p}$-almost periodic, i.e., that the Bochner transform $\hat{f} : {\mathbb R} \rightarrow L^{p}([0,1] : {\mathbb C})$ is not Lipschitz almost periodic (see \cite{stoja1} for the notion and more details). Suppose the contrary; then for each $\epsilon>0$ there exists a relatively dense set $R\subseteq {\mathbb R}$ such that, for every $\tau \in R$ and $t\in {\mathbb R}$, we have{\scriptsize 
\begin{align}\label{lips}
 \lim_{\delta \rightarrow 0+} \sup_{x\neq y;x,y\in [t-\delta,t+\delta]}\Biggl( \int^{1}_{0}\Biggl| \frac{\bigl[\sin \frac{1}{\zeta (s+x)}-\sin \frac{1}{\zeta (s+y)}\bigr]- \bigl[\sin \frac{1}{\zeta (s+x+\tau)}-\sin \frac{1}{\zeta(s+y+\tau)}\bigr]}{x-y} \Biggr|^{p} \, ds \Biggr)^{1/p}\leq \epsilon,
\end{align}}
where we have put $\zeta (t):=2+\cos t +\cos (\sqrt{2}t),$ $t\in {\mathbb R}.$ If $\tau \in R$ and $t\in {\mathbb R}$, then we can put $x=t$ and $y=t+\delta$ in \eqref{lips} in order to see that
\begin{align}\label{lips1}
 \lim_{\delta \rightarrow 0+}\Biggl( \int^{1}_{0}\Biggl| \frac{\bigl[\sin \frac{1}{\zeta (s+t)}-\sin \frac{1}{\zeta (s+t+\delta)}\bigr]- \bigl[\sin \frac{1}{\zeta (s+t+\tau)}-\sin \frac{1}{\zeta(s+t+\delta+\tau)}\bigr]}{\delta} \Biggr|^{p} \, ds \Biggr)^{1/p}\leq \epsilon.
\end{align}
Applying the Lagrange mean value theorem, we have
$$
\lim_{\delta \rightarrow 0+}\frac{\sin \frac{1}{\zeta (s+t)}-\sin \frac{1}{\zeta (s+t+\delta)}}{\delta}=-\frac{\zeta^{\prime}(s+t) }{\zeta^{2}(s+t)}\cos \frac{1}{\zeta (s+t)} ,\quad s\in [0,1],\ t\in {\mathbb R}.
$$
Keeping this equality and \eqref{lips1} in mind, the dominated convergence theorem implies that
\begin{align*}
\sup_{t\in {\mathbb R}}\Biggl( \int^{1}_{0}\Biggl| \frac{\zeta^{\prime}(s+t+\tau) }{\zeta^{2}(s+t+\tau)}\cos \frac{1}{\zeta (s+t+\tau)}-\frac{\zeta^{\prime}(s+t) }{\zeta^{2}(s+t)}\cos \frac{1}{\zeta (s+t)}  \Biggr|^{p}\, ds \Biggr)^{1/p}\leq \epsilon,
\end{align*}
which implies that the function 
$$
g(t):=\frac{\zeta^{\prime}(t) }{\zeta^{2}(t)}\cos \frac{1}{\zeta (t)}=\frac{\sin t+\sqrt{2}\sin(\sqrt{2}t)}{(2+\cos t+\cos (\sqrt{2}t))^{2}}\cos \frac{1}{2+\cos t+\cos (\sqrt{2}t)},\quad t\in {\mathbb R}
$$
is $S^{p}$-almost periodic. This cannot be true because the function $g(\cdot)$ is not Stepanov ($p$-)bounded, i.e.,
$$
\sup_{t\in {\mathbb R}}\int^{t+2\pi}_{t}|g(s)|\, ds=+\infty.
$$
Let us prove this. Applying the substitutions $v=\zeta(t)$ and $u=1/v$ after that, it readily follows that for each $t\in {\mathbb R}$ we have:
$$
\int^{t+2\pi}_{t}|g(s)|\, ds=\Biggl| \int^{\frac{1}{2+\cos t+\cos (\sqrt{2}t)}}_{\frac{1}{2+\cos (t+2\pi)+\cos (\sqrt{2}(t+2\pi))}} |\cos u|\, du \Biggr|.
$$
Therefore, there exist two finite real constants $c>0$ and $c_{1}>0$ such that
\begin{align}
\notag \int^{t+2\pi}_{t}& |g(s)|\, ds \geq c\Biggl \lfloor \frac{\Bigl|\frac{1}{2+\cos t+\cos (\sqrt{2}t)}-\frac{1}{2+\cos (t+2\pi)+\cos (\sqrt{2}(t+2\pi))}\Bigr|}{\pi/2} \Biggr \rfloor
\\\notag& = c\Biggl \lfloor \frac{4|\sin (\sqrt{2}\pi)| \cdot |\sin(\sqrt{2}(t+\pi))|}{\pi \cdot [2+\cos t+\cos (\sqrt{2}t)]\cdot [2+\cos (t+2\pi)+\cos (\sqrt{2}(t+2\pi))]}  \Biggr \rfloor
\\\label{mafina}& \geq c\Biggl \lfloor \frac{|\sin (\sqrt{2}\pi)| \cdot |\sin(\sqrt{2}(t+\pi))|}{\pi \cdot [2+\cos t+\cos (\sqrt{2}t)]}  \Biggr \rfloor \geq c_{1}\frac{|\sin(\sqrt{2}(t+\pi))|}{2+\cos t+\cos (\sqrt{2}t)},\quad t\rightarrow +\infty.
\end{align}
We proceed by using some arguments given in the remarkable paper \cite{nawrocki} by A. Nawrocki. Let $Q_{2m}$ be the odd natural number from the proof of \cite[Theorem 4]{nawrocki}; then we have $\lim_{m\rightarrow +\infty}Q_{2m}=+\infty$ and
\begin{align}\label{mafin1}
\frac{1}{2+\cos x_{m}+\cos (\sqrt{2}x_{m})}\geq 2\frac{x_{m}^{2}}{\pi^{4}},\quad m\in {\mathbb N},
\end{align}
where we have put $x_{m}:=Q_{2m}\pi$ ($m\in {\mathbb N}$). Suppose that $\sqrt{2}(Q_{2m}+1)\pi=k_{m}\pi +a_{m},$ where $k_{m}\in {\mathbb N}$ and $|a_{m}|\leq \pi/2$ ($m\in {\mathbb N}$). Since
$$
\Biggl| \sqrt{2} -\frac{k_{m}}{Q_{2m}+1}\Biggr|=\frac{|a_{m}|}{(Q_{2m}+1)\pi} ,\quad m\in {\mathbb N},
$$
the Liouville theorem (see, e.g., \cite[Theorem 1]{nawrocki}) implies the existence of a finite real number $d>0$ such that
$|a_{m}|\geq d\pi/(Q_{2m}+1)$ for all $m\in {\mathbb N}.$ Then we can apply the Jensen inequality in order to see that
\begin{align}\label{mafin12}
\bigl| \sin(a_{m}) \bigr|\geq \frac{2}{\pi}\bigl| a_{m} \bigr|\geq \frac{2d}{1+Q_{2m}}=\frac{2d}{1+(x_{m}/\pi)},\quad m\in {\mathbb N}.
\end{align}
Keeping in mind \eqref{mafin1}-\eqref{mafin12}, we get 
$$
\limsup_{m\rightarrow +\infty}\frac{|\sin(\sqrt{2}(x_{m}+\pi))|}{2+\cos x_{m}+\cos (\sqrt{2}x_{m})}=+\infty,
$$
contradicting \eqref{mafina}.

Finally, we would like to ask whether the function $f(\cdot)$ is $S^{p}$-almost periodic in variation, i.e., whether the function $\hat{f}(\cdot)$ is almost periodic in variation.
\end{example}

It is well known that any uniformly continuous Stepanov-$p$-almost periodic function $F: {\mathbb R}^{n}\rightarrow Y$ is almost periodic (\cite{nova-selected}). The interested reader may try to prove a metrical analogue of this result.

Before proceeding to the study of Weyl metrical approximations, we will state and prove the following simple result:

\begin{prop}\label{ns-bod}
Suppose that $F: \Lambda \times X \rightarrow {\mathbb C}$, $c\in {\mathbb C} \setminus \{0\}$ and the following conditions hold:
\begin{itemize}
\item[(i)] The function $\phi(\cdot)$ is monotonically increasing and there exists a function $\varphi : [0,\infty) \rightarrow [0,\infty)$ such that $\phi(xy)\leq \phi(x)\varphi(y)$ for all $x,\ y\geq 0.$
\item[(ii)] For every set $B\in {\mathcal B},$ there exists a finite real constant $c_{B}>0$ such that $|F({\bf t};x)|\geq c_{B},$ ${\bf t}\in \Lambda,$ $x\in B.$
\item[(iii)] The assumptions $0\leq f\leq g$ and $g\in P_{\Omega}$ ($g\in P$) imply $f\in P_{\Omega}$ ($f\in P$).
\item[(iv)] Condition \emph{(C1)} holds for $P_{\Omega}$ and $P.$
\end{itemize}
Then we have the following:
\begin{itemize}
\item[(a)] If $F(\cdot;\cdot)$ is Stepanov
$(\phi,{\mathrm R}, {\mathcal B},\phi,{\mathbb F},{\mathcal P})$-normal, then the function $(1/F)(\cdot;\cdot)$ is Stepanov
$(\phi,{\mathrm R}, {\mathcal B},\phi,{\mathbb F},{\mathcal P})$-normal.
\item[(b)] If $F\in S^{(\phi,{\mathbb F},c{\rm I},{\mathcal P}_{\Omega},{\mathcal P})}_{\Omega,\Lambda',{\mathcal B}}(\Lambda\times X :Y),$ then $(1/F)\in S^{(\phi,{\mathbb F},c^{-1}{\rm I},{\mathcal P}_{\Omega},{\mathcal P})}_{\Omega,\Lambda',{\mathcal B}}(\Lambda\times X :Y).$
\end{itemize}
\end{prop}

\begin{proof}
We will prove only (a). Due to (ii), we have $|F({\bf t};x)|\neq 0,$ ${\bf t}\in \Lambda,$ $x\in X.$
Let a set $B\in {\mathcal B}$ and a sequence $({\bf b}_{k})_{k\in {\mathbb N}}$ in ${\mathrm R}$ be given. Then we know that there exists a subsequence $({\bf b}_{k_{m}})_{m\in {\mathbb N}}$ of $({\bf b}_{k})_{k\in {\mathbb N}}$ such that, for every $\epsilon>0,$ there exists an integer $m_{0}\in {\mathbb N}$ such that, for every integers $m,\ m'\geq m_{0},$ we have \eqref{agape-laf1}. Denote by $d_{\Omega}>0$ and $d>0$ the corresponding constants from condition (C1) for the spaces $P_{\Omega}$ and $P,$ respectively. Then we have:{\scriptsize
\begin{align*}
&\sup_{x\in B}\Biggl\| \Bigl\| {\mathbb F}(\cdot,\cdot \cdot ) \phi\Bigl( \bigl| (1/F)(\cdot+\cdot \cdot+{\bf b}_{k_{m}};x)-(1/F)(\cdot+\cdot \cdot+{\bf b}_{k_{m'}};x)\bigr|\Bigr) \Bigr\|_{P_{\Omega}}\Biggr\|_{P}
\\& =\sup_{x\in B}\Biggl\| \Biggl\| {\mathbb F}(\cdot,\cdot \cdot ) \phi\Biggl( \frac{\bigl| F(\cdot+\cdot \cdot+{\bf b}_{k_{m}};x)-F(\cdot+\cdot \cdot+{\bf b}_{k_{m'}};x)\bigr|}{\bigl| F(\cdot+\cdot \cdot+{\bf b}_{k_{m}};x) \cdot F(\cdot+\cdot \cdot+{\bf b}_{k_{m'}};x)\bigr| }\Biggr) \Biggr\|_{P_{\Omega}}\Biggr\|_{P}
\\& \leq \sup_{x\in B}\Biggl\| \Bigl\| {\mathbb F}(\cdot,\cdot \cdot ) \varphi\bigl(c_{B}^{-2}\bigr)\phi\Bigl( \bigl| F(\cdot+\cdot \cdot+{\bf b}_{k_{m}};x)-F(\cdot+\cdot \cdot+{\bf b}_{k_{m'}};x)\bigr|\Bigr) \Bigr\|_{P_{\Omega}}\Biggr\|_{P}
\\& \leq \sup_{x\in B}\Biggl\| d_{\Omega}\bigl( 1+\varphi\bigl(c_{B}^{-2}\bigr)\bigr) \Bigl\| {\mathbb F}(\cdot,\cdot \cdot ) \phi\Bigl( \bigl| F(\cdot+\cdot \cdot+{\bf b}_{k_{m}};x)-F(\cdot+\cdot \cdot+{\bf b}_{k_{m'}};x)\bigr|\Bigr) \Bigr\|_{P_{\Omega}}\Biggr\|_{P}
\\ & \leq d\Bigl(1+ d_{\Omega}\bigl( 1+\varphi\bigl(c_{B}^{-2}\bigr)\bigr)\Bigr)\sup_{x\in B}\Biggl\| \Bigl\| {\mathbb F}(\cdot,\cdot \cdot ) \phi\Bigl( \bigl| F(\cdot+\cdot \cdot+{\bf b}_{k_{m}};x)-F(\cdot+\cdot \cdot+{\bf b}_{k_{m'}};x)\bigr|\Bigr) \Bigr\|_{P_{\Omega}}\Biggr\|_{P},
\end{align*}}
which simply implies the required.
\end{proof}

\noindent {\bf 2. Weyl metrical approximations}.
In this part, we will assume that condition (L) holds, where:\index{condition!(L)}
\begin{itemize}
\item[(L)] Let $\Omega$ be any compact subset of ${\mathbb R}^{n}$ with positive Lebesgue measure such that $\Lambda +\Omega \subseteq \Lambda.$ 
We assume that for each real number $l>0$ we have that
$P_{l} \subseteq [0,\infty)^{l\Omega},$  the zero function belongs to $P_{l}$, and 
${\mathcal P}_{l}=(P_{l},d_{l})$ is a pseudometric space. Let $P \subseteq [0,\infty)^{\Lambda},$ let the zero function belong to $P$, and let 
${\mathcal P}=(P,d)$ be a pseudometric space.
\end{itemize}

Now we are ready to
introduce the following notion:

\begin{defn}\label{weyl-approximation}\index{space!$e-{\mathcal B}-W^{p}(\Lambda \times X : Y)$}
Suppose that (L) holds, $\phi : [0,\infty) \rightarrow [0,\infty)$ and ${\mathbb F} : (0,\infty) \times  (\cup_{l>0}l\Omega) \times \Lambda \rightarrow (0,\infty).$
Then we say that the function $F(\cdot;\cdot)$ belongs to the class
$e-({\mathbb F},\phi,{\mathcal B})-W^{{\mathcal P}_{\cdot}}(\Lambda \times X : Y)$ [$e-({\mathbb F},\phi,{\mathcal B})-W^{{\mathcal P}_{\cdot}}_{\rho}(\Lambda \times X : Y);$ $e-({\mathbb F},\phi,{\mathcal B})_{j\in {\mathbb N}_{n}}-W^{{\mathcal P}_{\cdot}}_{\rho_{j}}(\Lambda \times X : Y)$] if and only if
for every $B\in {\mathcal B}$ and for every $\epsilon>0$ there exist a real number $l_{0}>0$ and a trigonometric polynomial [$\rho$-periodic function; $(\rho_{j})_{j\in {\mathbb N}_{n}}$-periodic function] $P(\cdot;\cdot)$ such that
\begin{align*}
\sup_{x\in B}\Biggl\| \Bigl\| {\mathbb F}(l,\cdot,\cdot \cdot) \phi\Bigl(\bigl\| P(\cdot \cdot +\cdot ;x)-F(\cdot \cdot+\cdot ;x) \bigr\|_{Y}\Bigr)\Bigr\|_{P_{l}}\Biggr\|_{P}<\epsilon,\quad l\geq l_{0}.
\end{align*}
\end{defn}\index{space!$e-({\mathbb F},\phi,{\mathcal B})-W^{{\mathcal P}_{\cdot}}(\Lambda \times X : Y)$}\index{space!$e-({\mathbb F},\phi,{\mathcal B})-W^{{\mathcal P}_{\cdot}}_{\rho}(\Lambda \times X : Y)$}\index{space!$e-({\mathbb F},\phi,{\mathcal B})_{j\in {\mathbb N}_{n}}-W^{{\mathcal P}_{\cdot}}_{\rho_{j}}(\Lambda \times X : Y)$}

The usual notion of class $e-{\mathcal B}-W^{{\mathcal P}_{\cdot}}(\Lambda \times X : Y)$,
introduced and analyzed in \cite[Subsection 6.3.1]{nova-selected}, is obtained by plugging 
${\mathbb F}(l,\cdot, \cdot \cdot)\equiv l^{-n/p},$
$P=L^{\infty}(\Lambda)$ and $P_{l}=L^{p}(l\Omega)$ for all $l>0.$ The classes $e-({\mathbb F},\phi,{\mathcal B})-W^{{\mathcal P}_{\cdot}}_{\omega,\rho}(\Lambda \times X : Y)$ and $e-({\mathbb F},\phi,{\mathcal B})_{j\in {\mathbb N}_{n}}-W^{{\mathcal P}_{\cdot}}_{\omega_{j},\rho_{j}}(\Lambda \times X : Y)$
have not been considered elsewhere even in the one-dimensional setting, with this choice of metric spaces.

Besides the notion introduced in Definition \ref{weyl-approximation}, we will consider the following ones (cf. also Definition \ref{stepanov-approximation1}):

\begin{defn}\label{weyl-approximation1}
Suppose that (L) holds, $\phi : [0,\infty) \rightarrow [0,\infty),$ and  ${\mathbb F} : (0,\infty) \times  (\cup_{l>0}l\Omega) \times \Lambda \rightarrow (0,\infty).$ 
\begin{itemize}
\item[(i)]\index{function!Weyl
$(\phi,{\mathrm R}, {\mathcal B},\phi,{\mathbb F},{\mathcal P})$-normal}
Suppose that ${\mathrm R}$ is any collection of sequences in $\Lambda''.$ Then we say that the function $F(\cdot;\cdot)$ is Weyl
$(\phi,{\mathrm R}, {\mathcal B},\phi,{\mathbb F},{\mathcal P})$-normal
if and only if for every set $B\in {\mathcal B}$ and for every sequence $({\bf b}_{k})_{k\in {\mathbb N}}$ in ${\mathrm R}$ there exists a subsequence $({\bf b}_{k_{m}})_{m\in {\mathbb N}}$ of $({\bf b}_{k})_{k\in {\mathbb N}}$ such that, for every $\epsilon>0,$ there exists an integer $m_{0}\in {\mathbb N}$ such that, for every integers $m,\ m'\geq m_{0},$ we have 
\begin{align}\label{agape-laf1}
\limsup_{l\rightarrow +\infty}\sup_{x\in B}\Biggl\| \Bigl\| {\mathbb F}(l,\cdot,\cdot \cdot ) \phi\Bigl( \bigl\| F(\cdot+\cdot \cdot+{\bf b}_{k_{m}};x)-F(\cdot+\cdot \cdot+{\bf b}_{k_{m'}};x)\bigr\|_{Y}\Bigr) \Bigr\|_{P_{l}}\Biggr\|_{P}<\epsilon .
\end{align}
\item[(ii)] By $e-W^{(\phi,{\mathbb F},\rho,{\mathcal P}_{\Omega},{\mathcal P})}_{\Omega,\Lambda',{\mathcal B}}(\Lambda\times X :Y)$ we denote the set consisting of all functions $F : \Lambda \times X \rightarrow Y$ such that, for every $\epsilon>0$ and $B\in {\mathcal B},$ there exist two finite real numbers
$l>0$
and
$L>0$ such that for each ${\bf t}_{0}\in \Lambda'$ there exists $\tau \in B({\bf t}_{0},L)\cap \Lambda'$ such that, for every $x\in B,$ the mapping ${\bf u}\mapsto G_{x}({\bf u})\in \rho( F({\bf u};x)),$ ${\bf u}\in  (\cup_{l>0}l\Omega) +\Lambda$ is well defined, and
\begin{align*}
\sup_{x\in B}\Biggl\|{\Bigl\| \mathbb F}(l,\cdot,\cdot \cdot)\phi\Bigl( \bigl\| F({\bf \tau}+\cdot+\cdot \cdot;x)-G_{x}(\cdot+\cdot \cdot)\bigr\|_{Y}\Bigr) \Bigr\|_{P_{l}} \Biggr\|_{P}<\epsilon.
\end{align*}\index{space!$e-W^{(\phi,{\mathbb F},\rho,{\mathcal P}_{\Omega},{\mathcal P})}_{\Omega,\Lambda',{\mathcal B}}(\Lambda\times X :Y)$}
\item[(iii)] By $W^{(\phi,{\mathbb F},\rho,{\mathcal P}_{\Omega},{\mathcal P})}_{\Omega,\Lambda',{\mathcal B}}(\Lambda\times X :Y)$ we denote the set consisting of all functions $F : \Lambda \times X \rightarrow Y$ such that, for every $\epsilon>0$ and $B\in {\mathcal B},$ there exists a finite real number
$L>0$ such that for each ${\bf t}_{0}\in \Lambda'$ there exists $\tau \in B({\bf t}_{0},L)\cap \Lambda'$ such that, for every $x\in B,$ the mapping ${\bf u}\mapsto G_{x}({\bf u})\in \rho( F({\bf u};x)),$ ${\bf u}\in  (\cup_{l>0}l\Omega) +\Lambda$ is well defined, and
\begin{align*}
\limsup_{l\rightarrow +\infty}\sup_{x\in B}\Biggl\|{\Bigl\| \mathbb F}(l,\cdot,\cdot \cdot)\phi\Bigl( \bigl\| F({\bf \tau}+\cdot+\cdot \cdot;x)-G_{x}(\cdot+\cdot \cdot)\bigr\|_{Y}\Bigr) \Bigr\|_{P_{l}} \Biggr\|_{P}<\epsilon.
\end{align*}
\end{itemize}
\end{defn}\index{space!$W^{(\phi,{\mathbb F},\rho,{\mathcal P}_{\Omega},{\mathcal P})}_{\Omega,\Lambda',{\mathcal B}}(\Lambda\times X :Y)$}

The notion introduced in Definition \ref{weyl-approximation} is stronger than the notion introduced in Definition \ref{weyl-approximation1};
the interested reader may simply clarify some reasons justifying this (cf. also Proposition \ref{rep}). As already emphasized for the Stepanov metrical approximations, a uniformly recurrent analogue of the notion introduced in Definition \ref{weyl-approximation1}(ii)-(iii) can be also analyzed.

Concerning the linear structure of introduced function spaces, we will only note the following:  Suppose that the function $\phi(\cdot)$ is monotonically increasing and there exists a finite real constant $c>0$ such that $\phi(x+y)\leq c[\phi(x)+\phi(y)]$ for all $x,\ y\geq 0.$ Then we have the following:
\begin{itemize}
\item[(i)]
Equipped with the usual operations, the set of all strongly $(\phi,{\mathbb F},{\mathcal B},{\mathcal P})$-almost periodic functions forms a vector space provided that condition (C2) holds.
\item[(ii)]  Equipped with the usual operations, the set $e-({\mathbb F},{\mathcal B})-S^{{\mathcal P}_{\Omega}}(\Lambda \times X : Y)$ forms a vector space provided that condition (C2) holds and condition (C2) holds for the pseudometric space $P_{\Omega}.$
\item[(iii)] Equipped with the usual operations, the set $e-({\mathbb F},\phi,{\mathcal B})-W^{{\mathcal P}_{\cdot}}(\Lambda \times X : Y)$ forms a vector space provided that condition (C2) holds and condition (C2) holds for the pseudometric space $P_{l\Omega},$ for any $l>0.$
\end{itemize}

The interested reader may try to reformulate the statements of \cite[Proposition 3.5, Corollary 3.6]{metrical-weyl} in our new context. Concerning the embeddings of Stepanov classes introduced in this paper into the equi-Weyl classes of functions introduced in this paper, we will clarify the following result, only:

\begin{prop}\label{carlos}
Suppose that $\Lambda={\mathbb R}^{n}$ and the following conditions hold:
\begin{itemize}
\item[(i)] $P_{1}$ is a Banach space and condition \emph{(C0)} holds for $P_{1}.$
\item[(ii)] $P_{\Omega}=L^{p}_{\nu}(\Omega : [0,\infty))$ for some $p\in [1,\infty)$  and a Lebesgue measurable function $\nu : \Omega \rightarrow (0,\infty);$ $P=C_{0,w}(\Lambda : [0,\infty))$ for some function $w: \Lambda \rightarrow (0,\infty)$ such that the function $1/w(\cdot)$ is locally bounded.
\item[(iii)] There exists a real number $l_{0}>0$ such that, for every $l\geq l_{0},$ we have
\begin{align*}
\Biggl\| \sum_{k\in ({\mathbb N}_{0}^{\lceil l \rceil-1})^{n}}\frac{1}{w(\cdot \cdot +k)}\Biggr\|_{P_{1}}\leq \mbox{Const.}
\end{align*}
\item[(iv)] There exists a real number $l_{0}>0$ such that, for every $l\geq l_{0},$ we have
\begin{align*}
\Bigl\| {\mathbb F}_{1}(l,\cdot,\cdot \cdot) W(\cdot,\cdot \cdot;x) \Bigr\|_{P_{l}^{1}}\leq \sum_{k\in ({\mathbb N}_{0}^{\lceil l \rceil-1})^{n}} \Bigl\| {\mathbb F}(\cdot,\cdot \cdot +k)W(\cdot,\cdot \cdot +k;x)\Bigr\|_{P_{\Omega}},
\end{align*}
provided that the above terms are well-defined.
\end{itemize}
If $F\in e-({\mathbb F},{\mathcal B})-S^{{\mathcal P}_{\Omega}}(\Lambda \times X : Y)$, then $F\in e-({\mathbb F}_{1},\phi,{\mathcal B})-W^{{\mathcal P}_{\cdot}^{1}}(\Lambda \times X : Y).$
\end{prop}

\begin{proof}
The proof simply follows from the prescribed assumptions and the next computation (the meaning of inequality in the fourth line of computation is clear from the context):{\scriptsize
\begin{align*}
&\sup_{x\in B}\Biggl\| \Bigl\| {\mathbb F}_{1}(l,\cdot,\cdot \cdot) \phi\Bigl(\bigl\| P(\cdot \cdot +\cdot ;x)-F(\cdot \cdot+\cdot ;x) \bigr\|_{Y}\Bigr)\Bigr\|_{P_{l^{1}}}\Biggr\|_{P^{1}}
\\& \leq j\sup_{x\in B}\Biggl\| \Bigl\| \sum_{k\in ({\mathbb N}_{0}^{\lceil l \rceil-1})^{n}} \Bigl\| {\mathbb F}(\cdot,\cdot \cdot +k) \phi\Bigl(\bigl\| P(\cdot \cdot +k+\cdot ;x)-F(\cdot \cdot+k+\cdot ;x) \bigr\|_{Y}\Bigr)\Bigr\|_{P_{\Omega}}\Biggr\|_{P^{1}}
\\& = j\sup_{x\in B}\Biggl\|  \sum_{k\in ({\mathbb N}_{0}^{\lceil l \rceil-1})^{n}}\Biggl( \int_{\Omega}\Bigl|{\mathbb F}({\bf u},\cdot \cdot +k) \phi\Bigl(\bigl\| P(\cdot \cdot +k+{\bf u} ;x)-F(\cdot \cdot+k+{\bf u} ;x) \bigr\|_{Y}\Bigr)\Bigr|^{p}\nu^{p}({\bf u})\, d{\bf u}\Biggr)^{1/p}\Biggr\|_{P^{1}}
\\& \leq j\sup_{x\in B}\Biggl\|   \sum_{k\in ({\mathbb N}_{0}^{\lceil l \rceil-1})^{n}}\frac{\epsilon}{w(\cdot \cdot +k)}  \Biggr\|_{P^{1}}\leq j\cdot \epsilon \cdot \mbox{Const.}\, .
\end{align*}}
\end{proof}

\subsection{Besicovitch and Doss metrical approximations}\label{jet}

In this subsection, we will assume that condition (B) holds, where: \index{condition!(B)}
\begin{itemize}
\item[(B)] For each real number $t>0$ we have 
$P_{t} \subseteq [0,\infty)^{\Lambda_{t}},$  the zero function belongs to $P_{t}$, and 
${\mathcal P}_{t}=(P_{t},d_{t})$ is a pseudometric space. 
\end{itemize}

We are ready to
introduce the following notion:

\begin{defn}\label{ebe}
Suppose that $F: \Lambda \times X \rightarrow Y,$ $\phi : [0,\infty) \rightarrow [0,\infty)$ and ${\mathbb F} : (0,\infty) \times \Lambda \rightarrow (0,\infty).$ Then we say that the function $F(\cdot;\cdot)$ belongs to the class $e-({\mathcal B},\phi,{\mathbb F})-B^{{\mathcal P}_{\cdot}}(\Lambda \times X : Y)$ [$e-({\mathcal B},\phi,{\mathbb F})-B^{{\mathcal P}_{\cdot}}_{\rho}(\Lambda \times X : Y);$ $e-({\mathcal B},\phi,{\mathbb F})_{j\in {\mathbb N}_{n}}-B^{{\mathcal P}_{\cdot}}_{\rho_{j}}(\Lambda \times X : Y)$] if and only if for each set $B\in {\mathcal B}$ there exists a sequence $(P_{k}(\cdot;\cdot))$ of trigonometric polynomials [$\rho$-periodic functions; $(\rho_{j})_{j\in {\mathbb N}_{n}}$-periodic functions] such that
\begin{align}\label{jew}
\lim_{k\rightarrow +\infty}\limsup_{t\rightarrow +\infty}\sup_{x\in B}\Bigl\| {\mathbb F}(t,\cdot) \phi\Bigl( \bigl\|F(\cdot;x)-P_{k}(\cdot;x)\bigr\|_{Y}\Bigr) \Bigr\|_{P_{t}}=0.
\end{align}
If $\phi(x)\equiv x,$ then we omit the term ``$\phi$'' from the notation; if $X=\{0\},$ then we omit the term ``${\mathcal B}$'' from the notation.
\end{defn}\index{space!$e-({\mathcal B},\phi,{\mathbb F})-B^{{\mathcal P}_{\cdot}}(\Lambda \times X : Y)$} \index{space!$e-({\mathcal B},\phi,{\mathbb F})-B^{{\mathcal P}_{\cdot}}_{\rho}(\Lambda \times X : Y)$} \index{space!$e-({\mathcal B},\phi,{\mathbb F})_{j\in {\mathbb N}_{n}}-B^{{\mathcal P}_{\cdot}}_{\rho_{j}}(\Lambda \times X : Y)$}

Immediately from definition, it follows that, for every $F\in e-({\mathcal B},\phi,{\mathbb F})-B^{{\mathcal P}_{\cdot}}(\Lambda \times X : Y)$ and $\lambda \in {\mathbb R}^{n},$ we have $e^{i \langle \lambda , \cdot \rangle}F\in e-({\mathcal B},\phi,{\mathbb F})-B^{{\mathcal P}_{\cdot}}(\Lambda \times X : Y);$ this also holds for the corresponding classes introduced in Definition \ref{strong-app}, Definition \ref{stepanov-approximation} and Definition \ref{weyl-approximation}. The class $e-({\mathcal B},\phi,{\mathbb F})-B^{p(\cdot)}(\Lambda \times X : Y)$ considered in \cite{multi-besik} is nothing else but the class $e-({\mathcal B},\phi,{\mathbb F})-B^{{\mathcal P}_{\cdot}}(\Lambda \times X : Y)$ with ${\mathbb F}(t,\cdot)\equiv {\mathbb F}(t)$ and $P_{t}=L^{p(\cdot)}(\Lambda_{t})$ for all $t>0.$ The classes $e-({\mathcal B},\phi,{\mathbb F})-B^{{\mathcal P}_{\cdot}}_{\omega,\rho}(\Lambda \times X : Y)$ and $e-({\mathcal B},\phi,{\mathbb F})_{j\in {\mathbb N}_{n}}-B^{{\mathcal P}_{\cdot}}_{\omega_{j},\rho_{j}}(\Lambda \times X : Y)$
have not been considered elsewhere even in the one-dimensional setting, with this choice of metric spaces.

We continue this part by extending the notion of Besicovitch$-({\mathrm R}, {\mathcal B},\phi,{\mathbb F})-B^{p(\cdot)}$-normality, introduced recently in \cite[Definition 2.11]{multi-besik}:\index{function!Besicovitch
$(\phi,{\mathrm R}, {\mathcal B},\phi,{\mathbb F},{\mathcal P})$-normal}

\begin{defn}\label{petmet}
Suppose that ${\mathrm R}$ is any collection of sequences in $\Lambda'',$ $F: \Lambda \times X \rightarrow Y,$ $\phi : [0,\infty) \rightarrow [0,\infty)$ and ${\mathbb F} : (0,\infty) \rightarrow (0,\infty).$ Then we say that the function $F(\cdot;\cdot)$ is Besicovitch
$(\phi,{\mathrm R}, {\mathcal B},\phi,{\mathbb F},{\mathcal P})$-normal
if and only if for every set $B\in {\mathcal B}$ and for every sequence $({\bf b}_{k})_{k\in {\mathbb N}}$ in ${\mathrm R}$ there exists a subsequence $({\bf b}_{k_{m}})_{m\in {\mathbb N}}$ of $({\bf b}_{k})_{k\in {\mathbb N}}$ such that, for every $\epsilon>0,$ there exists an integer $m_{0}\in {\mathbb N}$ such that, for every integers $m,\ m'\geq m_{0},$ we have 
\begin{align*}
\limsup_{t\rightarrow +\infty}\sup_{x\in B}\Biggl\|{\mathbb F}(t,\cdot)\phi\Bigl( \bigl\| F(\cdot+{\bf b}_{k_{m}};x)-F(\cdot+{\bf b}_{k_{m'}};x)\bigr\|_{Y} \Bigr)\Biggr\|_{P_{t}} <\epsilon.
\end{align*}
\end{defn}

The interested reader may try to clarify some assumptions under which any function $e-({\mathbb F},\phi,{\mathcal B})-W^{{\mathcal P}_{\cdot}}(\Lambda \times X : Y)$ [$e-({\mathbb F},\phi,{\mathcal B})_{j\in {\mathbb N}_{n}}-W^{{\mathcal P}_{\cdot}}_{{\rm I}}(\Lambda \times X : Y)$] is Besicovitch
$(\phi,{\mathrm R}, {\mathcal B},\phi,{\mathbb F},{\mathcal P})$-normal; as is well known, the converse statement does not hold (cf. \cite[Proposition 2.12, Example 2.15]{multi-besik} and Proposition \ref{rep}). A metrical analogue of \cite[Proposition 2.13]{multi-besik} can be formulated without any substantial difficulties; for simplicity, we will not consider here any Bohr analogue of the notion introduced in Definition \ref{ebe} and Definition \ref{petmet}.

The following notion generalizes the notion introduced recently in \cite[Definition 1]{doss-rn}: 

\begin{defn}\label{prespansko} 
Let $\phi : [0,\infty) \rightarrow [0,\infty)$ and ${\mathbb F} : (0,\infty) \times \Lambda \rightarrow (0,\infty).$
\begin{itemize}
\item[(i)] Suppose that the function $F : \Lambda \times X \rightarrow Y$ satisfies that ${\mathbb F}(t;\cdot)\phi(\| F(\cdot;x)\|_{Y})\in P_{t}$ for all $t>0$ and $x\in X.$ Then we say that the function $F(\cdot;\cdot)$ is Besicovitch-$({\mathcal P},\phi,{\mathbb F},{\mathcal B})$-bounded if and only if, for every $B\in {\mathcal B},$ there exists a finite real number $M_{B}>0$ such that
\begin{align*}
\limsup_{t\rightarrow +\infty}\sup_{x\in B}\Bigl\|{\mathbb F}(t;\cdot)\phi \bigl(\| F(\cdot ;x)\|_{Y}\bigr)\Bigr\|_{P_{t}} \leq M_{B}.
\end{align*}
\item[(ii)] Suppose that the function $F : \Lambda \times X \rightarrow Y$ satisfies that ${\mathbb F}(t;\cdot)\phi(\| F(\cdot +\tau ;x)-y_{\cdot;x}\|_{Y})\in P_{t}$ for all $t>0,$ $x\in X,$ $\tau \in \Lambda'$
and $y_{\cdot;x}\in \rho(F(\cdot;x)).$
\begin{itemize}
\item[(a)] We say that the function $F : \Lambda \times X \rightarrow Y$ is Besicovitch-$({\mathcal P},\phi,{\mathbb F},{\mathcal B},\Lambda',\rho)$-continuous if and only if, for every $B\in {\mathcal B}$ as well as for every $t>0,$ $x\in B$ and $\cdot \in \Lambda_{t},$ we have the existence of an element $y_{\cdot;x}\in \rho (F(\cdot;x))$ such that
\begin{align*}
\lim_{\tau \rightarrow 0,\tau \in \Lambda'}\limsup_{t\rightarrow +\infty}\sup_{x\in B}\Bigl\| {\mathbb F}(t;\cdot)\phi \bigl(\| F(\cdot +\tau ;x)-y_{\cdot;x}\|_{Y}\bigr)\Bigr\|_{P_{t}}=0.
\end{align*}
\item[(b)] We say that the function $F(\cdot;\cdot)$ is Doss-$({\mathcal P},\phi,{\mathbb F},{\mathcal B},\Lambda',\rho)$-almost periodic if and only if,
for every $B\in {\mathcal B}$ and $\epsilon>0,$ there exists $l>0$ such that for each ${\bf t}_{0}\in \Lambda'$ there exists a point $\tau \in B({\bf t}_{0},l) \cap \Lambda'$ such that, for every $t>0,$ $x\in B$ and $\cdot \in \Lambda_{t},$ we have the existence of an element $y_{\cdot;x}\in \rho (F(\cdot;x))$ such that
\begin{align*}
\limsup_{t\rightarrow +\infty}\sup_{x\in B}\Bigl\|{\mathbb F}(t;\cdot)\phi \bigl(\| F(\cdot +\tau ;x)-y_{\cdot;x}\|_{Y}\bigr)\Bigr\|_{P_{t}}<\epsilon.
\end{align*}
\item[(c)] We say that the function $F(\cdot;\cdot)$ is  Doss-$({\mathcal P},\phi,{\mathbb F},{\mathcal B},\Lambda',\rho)$-uniformly recurrent if and only if,
for every $B\in {\mathcal B},$ there exists a sequence $(\tau_{k})\in \Lambda'$ such that, for every $t>0,$ $x\in B$ and $\cdot \in \Lambda_{t},$ we have the existence of an element $y_{\cdot;x}\in \rho (F(\cdot;x))$ such that
\begin{align*}
\lim_{k\rightarrow +\infty}\limsup_{t\rightarrow +\infty}\sup_{x\in B}\Bigl\| {\mathbb F}(t;\cdot) \phi \bigl(\| F(\cdot +\tau_{k} ;x)-y_{\cdot;x}\|_{Y}\bigr)\Bigr\|_{P_{t}}=0.
\end{align*}
\end{itemize}
\end{itemize}
\end{defn}\index{function!Besicovitch-$({\mathcal P},\phi,{\mathbb F},{\mathcal B})$-bounded}\index{function!Besicovitch-$({\mathcal P},\phi,{\mathbb F},{\mathcal B},\Lambda',\rho)$-continuous}\index{function!Doss-$({\mathcal P},\phi,{\mathbb F},{\mathcal B},\Lambda',\rho)$-almost periodic}
\index{function!Doss-$({\mathcal P},\phi,{\mathbb F},{\mathcal B},\Lambda',\rho)$-uniformly recurrent}

Let us note that the statement of Proposition \ref{nbm} can be simply reformulated for Doss-$({\mathcal P},\phi,{\mathbb F},{\mathcal B},\Lambda',\rho)$-almost periodic type functions (cf. also \cite[Proposition 1]{doss-rn}).
Furthermore, the following result can be deduced following the lines of proof of \cite[Proposition 2.4]{multi-besik} (for simplicity, we assume here that the function ${\mathbb F}$ does not depend on the second argument):

\begin{prop}\label{anat-metrical}
Suppose that ${\mathcal B}$ consists of bounded subsets of $X$, $F : \Lambda \times X \rightarrow Y$ and, for every fixed element $x\in X,$ the function $F(\cdot;x)$ is Lebesgue measurable. Suppose, further, that the function $\phi(\cdot)$ is monotonically increasing and there exists a finite real constant $c>0$ such that $\phi(x+y)\leq c[\phi(x)+\phi(y)]$ for all $x,\ y\geq 0.$
\begin{itemize}
\item[(i)]
Let the following conditions hold:
\begin{itemize}
\item[(a)] For every $t>0,$ $P_{t}$ contains all positive constants, and for every real number $d>0,$ there exist two real numbers $d'>0$ and $t_{0}>0$ such that, for every $t\geq t_{0},$ we have $\|d\|_{P_{t}}\leq d'.$
\item[(b)] There exist two real numbers $t_{1}>0$ and $M>0$ such that ${\mathbb F}(t)\| 1\|_{P_{t}}\leq M,$ $t\geq t_{1}.$
\item[(c)]  There exist two real numbers $t_{2}>0$ such that condition \emph{(C2)} holds for any pseudometric space $P_{t},$ $t\geq t_{2}.$ 
\end{itemize}
Then, any function $F\in e-({\mathcal B},\phi,{\mathbb F})-B^{{\mathcal P}_{\cdot}}(\Lambda \times X : Y)$ is Besicovitch-$({\mathcal P},\phi,{\mathbb F},{\mathcal B})$-bounded. 
\item[(ii)] Suppose that 
$\phi(\cdot)$ is continuous at the point $t=0$.
Let the following conditions hold:
\begin{itemize}
\item[(a)] For every $t>0,$ the pseudometric space $P_{t}$ contains all positive constants and $\lim_{\epsilon \rightarrow 0+}\| \epsilon\|_{P_{t}}=0.$
\item[(b)] There exist two real numbers $t_{1}>0$ and $c>0$ such that condition \emph{(C0)} holds for $P_{t},$ $t\geq t_{1}.$
\item[(c)] There exist two real numbers $t_{2}>0$ such that condition \emph{(C3)} holds for any pseudometric space $P_{t},$ $t\geq t_{2}.$ 
\item[(d)] There exist two finite real constants $c>0$ and $t_{2}>0$ such that, for every $t\geq t_{2},$ the assumption $F(\cdot;\cdot)\in P_{t}$ implies $F(\cdot +\tau;\cdot) \in P_{t}$ for all $\tau \in \Lambda'$ with $|\tau|\leq 1,$ and $\| F(\cdot +\tau;x)\|_{P_{t}} \leq c\|F(\cdot;x)\|_{P_{t}}$ for all $t\geq t_{2},$ $x\in X$ and $\tau \in \Lambda'$ with $|\tau|\leq 1.$
\end{itemize}
Then, any function $F\in e-({\mathcal B},\phi,{\mathbb F})-B^{{\mathcal P}_{\cdot}}(\Lambda \times X : Y)$ is Besicovitch-$({\mathcal P},\phi,{\mathbb F},{\mathcal B},\Lambda',{\rm I})$-continuous for any set $\Lambda' \subseteq \Lambda''. $
\item[(iii)] Suppose that $\phi(\cdot)$ is continuous at the point $t=0$, conditions \emph{(a)-(c)} of \emph{(ii)} hold and condition \emph{(d)} of \emph{(ii)} holds for every $\tau \in \Lambda'=\Lambda.$
Let
$F\in e-({\mathcal B},\phi,{\mathbb F})-B^{{\mathcal P}_{\cdot}}(\Lambda \times X : Y)$.
Then the following holds:
\begin{itemize}
\item[(a)]
The function $F(\cdot;\cdot)$ is Doss-$({\mathcal P},\phi,{\mathbb F},{\mathcal B},\Lambda,{\rm I})$-almost periodic, provided that $\Lambda+\Lambda \subseteq \Lambda$ and, for every points $( t_{1},..., t_{n} )\in  \Lambda$ and $( \tau_{1},..., \tau_{n} )\in  \Lambda,$ 
the points $( t_{1},t_{2}+\tau_{2},..., t_{n}+\tau_{n} ),$ $( t_{1},t_{2},t_{3}+\tau_{3},..., t_{n}+\tau_{n} ),...,$ $( t_{1},t_{2},..., t_{n-1}, t_{n}+\tau_{n} ),$ also belong to $ \Lambda .$
\item[(b)] The function $F(\cdot;\cdot)$ is Doss-$({\mathcal P},\phi,{\mathbb F},{\mathcal B},\Lambda \cap \Delta_{n},{\rm I})$-almost periodic, provided that 
$\Lambda \cap \Delta_{n}\neq \emptyset,$ $\Lambda+(\Lambda\cap \Delta_{n})\subseteq \Lambda$ and that, for every points $( t_{1},..., t_{n} )\in \Lambda$ and $( \tau,..., \tau )\in \Lambda\cap \Delta_{n},$ 
the points $( t_{1},t_{2}+\tau,..., t_{n}+\tau ),$ $( t_{1},t_{2},t_{3}+\tau,..., t_{n}+\tau ),... ,$ $( t_{1},t_{2},..., t_{n-1}, t_{n}+\tau ),$ also belong to $\Lambda\cap \Delta_{n}.$
\end{itemize}
\end{itemize}
\end{prop}

Suppose that ${\mathcal B}$ consists of bounded subsets of $X$, $F,\ G : \Lambda \times X \rightarrow Y$, the function $\phi(\cdot)$ is monotonically increasing, $\phi(0)=0$ and $\phi(x+y)\leq \phi(x)+\phi(y)$ for all $x,\ y\geq 0.$ Let conditions (a)-(b) given in the formulation of
Proposition \ref{anat-metrical}(i), and
let condition (c) holds with $d=1,$ 
for the both functions
$F$ and $G.$ Equipped with the usual operations, the 
set $e-({\mathcal B},\phi,{\mathbb F})-B^{{\mathcal P}_{\cdot}}(\Lambda \times X : Y)$ forms a vector space on account of condition (c). Therefore, the
functions $F(\cdot;\cdot)$ and $G(\cdot;\cdot)$ are Besicovitch-$({\mathcal P},\phi,{\mathbb F},{\mathcal B})$-bounded. Let a set $B\in {\mathcal B}$ be fixed. Then 
$$
d_{B}(F,G):=\limsup_{t\rightarrow +\infty}\sup_{x\in B}\Bigl\| {\mathbb F}(t;\cdot)\phi\Bigl( \bigl\| F({\cdot};x)-G({\cdot};x) \bigr\|_{Y} \Bigr)\Bigr\|_{P_{t}}
$$
defines a pseudometric on the set $
e-({\mathcal B},\phi,{\mathbb F})-B^{{\mathcal P}_{\cdot}}(\Lambda \times X : Y).$ Using the idea from the original proof of J. Marcinkiewicz (see also \cite[pp. 249--252]{188}), we have recently proved that the pseudometric space $(e-({\mathcal B},\phi,{\mathbb F})-B^{p(\cdot)}(\Lambda \times X : Y),d_{B})$ is complete (see \cite[Theorem 2.5]{multi-besik}). We will not consider here some sufficient conditions ensuring that the pseudometric space  
$
e-({\mathcal B},\phi,{\mathbb F})-B^{{\mathcal P}_{\cdot}}(\Lambda \times X : Y)$ is complete.

In \cite[Subsection 2.1]{doss-rn}, we have investigated the relationship between the Doss almost periodiicty and Weyl almost periodicity and proved especially that every Weyl-$p$-almost periodic function $f : {\mathbb R}\rightarrow Y$ is Doss-$p$-almost periodic ($1\leq p<+\infty$). The multi-dimensional analogue of this statement has been clarified in \cite[Proposition 8]{doss-rn} and we will only note that it could be interesting to formulate this result in the metrical framework.

We continue our work by emphasizing the following facts:
\begin{itemize}
\item[(BE1)] The reader can simply clarify some natural conditions under which the introduced spaces of functions are translation invariant. Suppose, for example, that $\tau \in \Lambda'',$ $x_{0}\in X$ and $F\in e-({\mathcal B},\phi,{\mathbb F})-B^{{\mathcal P}_{\cdot}}(\Lambda \times X : Y)$ [$e-({\mathcal B},\phi,{\mathbb F})-B^{{\mathcal P}_{\cdot}}_{\rho}(\Lambda \times X : Y);$ $e-({\mathcal B},\phi,{\mathbb F})_{j\in {\mathbb N}_{n}}-B^{{\mathcal P}_{\cdot}}_{\rho_{j}}(\Lambda \times X : Y)$]. Then we have $F(\cdot +\tau;\cdot+x_{0})\in e-({\mathcal B}_{x_{0}},\phi,{\mathbb F})-B^{{\mathcal P}_{\cdot}}(\Lambda \times X : Y)$ [$e-({\mathcal B},\phi,{\mathbb F})-B^{{\mathcal P}_{\cdot}}_{\rho}(\Lambda \times X : Y);$ $e-({\mathcal B},\phi,{\mathbb F})_{j\in {\mathbb N}_{n}}-B^{{\mathcal P}_{\cdot}}_{\rho_{j}}(\Lambda \times X : Y)$] with ${\mathcal B}_{x_{0}}\equiv \{-x_{0}+B : B\in {\mathcal B}\},$ provided that there exist
a finite real constant $c_{\tau}>0$ such that, for every $x\in B$ and $t>0,$ the assumption ${\mathbb F}(t+|\tau|,\cdot)\phi(\| H(\cdot,x)\|_{Y}) \in P_{t+|\tau|}$ implies
${\mathbb F}(t,\cdot)\phi(\| H(\cdot+\tau,x)\|_{Y}) \in P_{t}$ and 
$$
\Bigl\| {\mathbb F}(t,\cdot)\phi \bigl(\| H(\cdot+\tau,x)\|_{Y}\bigr) \Bigr\|_{P_{t}}\leq c_{\tau}\Bigl\|{\mathbb F}(t+|\tau|,\cdot)\phi\bigl(\| H(\cdot,x)\|_{Y}\bigr) \Bigr\|_{P_{t+|\tau|}}.
$$
\item[(BE2)] Suppose that the function $\phi(\cdot)$ is monotonically increasing, continuous at the point zero, there exists a finite real constant $c>0$ such that $\phi(x+y)\leq c[\phi(x)+\phi(y)]$ for all $x,\ y\geq 0,$ the mapping $t\mapsto\| {\mathbb F}(t,\cdot)\|_{P_{t}},$ $t>0$ is bounded at plus infinity, and for each $t>0$ we have that $P_{t}$ has a linear vector structure and contains all positive constants. Then $F\in e-({\mathcal B},\phi,{\mathbb F})-B^{{\mathcal P}_{\cdot}}(\Lambda \times X : Y)$ if and only if for each set $B\in {\mathcal B}$ there exists a sequence $(F_{k}(\cdot;\cdot))$ of strongly ${\mathcal B}$-almost periodic functions such that \eqref{jew} holds with the polynomial $P_{k}(\cdot;\cdot)$ replaced therein with the function $F_{k}(\cdot;\cdot).$ We can similarly consider the corresponding question for the classes  introduced in Definition \ref{strong-app}, Definition \ref{stepanov-approximation} and Definition \ref{weyl-approximation}.
\end{itemize}

The pointwise multiplication of multi-dimensional Besicovitch almost periodic type functions have recently been analyzed in \cite[Proposition 2.7, Proposition 2.14]{multi-besik}. Here we will clarify only one result concerning pointwise multiplication of functions obtained as metrical approximations by trigonometric polynomials or $\rho$-periodic type functions (the proof is very similar to the proof of the above-mentioned statement and therefore omitted):

\begin{prop}\label{approxb}
Suppose that the following conditions hold:
\begin{itemize}
\item[(i)]
There exists a finite real constant $c>0$ such that $\phi(x+y)\leq c[\phi(x)+\phi(y)]$ for all $x,\ y\geq 0,$ and there exists a function $\varphi : [0,\infty) \rightarrow [0,\infty)$ such that $\phi(xy)\leq \varphi(x)\phi(y)$ for all $x,\ y\geq 0.$
\item[(ii)] Condition \emph{(B)} holds with the pseudometric spaces ${\mathcal P}_{t}^{1}=(P_{t}^{1},d_{t}^{1})$ and ${\mathcal P}_{t}^{2}=(P_{t}^{2},d_{t}^{2})$ as well as the assumptions $f\in P_{t}^{1}$ and $g\in P_{t}^{2}$ imply $f\cdot g \in P_{t}$  ($t>0$) and the existence of a finite real constant $k>0$ such that 
\begin{align}\label{helder-aps}
\bigl\| f\cdot g \bigr\|_{P_{t}}\leq k\bigl\| f\bigl\|_{P_{t}^{1}}\cdot \bigl\| g\bigl\|_{P_{t}^{2}},\quad t>0,\ f\in P_{t}^{1},\ g\in P_{t}^{2}.
\end{align}
\item[(iii)] Condition \emph{(C0)} holds for $P_{t}$, provided that $t$ is sufficiently large.
\item[(iv)] For every set $B\in {\mathcal B},$ there exists a real number $t_{0}>0$ such that, for every $x\in B,$ the function ${\mathbb F}_{1}(t,\cdot)\varphi(|P(\cdot;x)|)$ 
belongs to $P_{t}^{1}$ for any $t>0$ and any trigonometric polynomial [$c$-periodic function; $(c_{j})_{j\in {\mathbb N}_{n}}$-periodic function],
the function ${\mathbb F}_{2}(t,\cdot)\varphi(\|G(\cdot;x)\|_{Y})$
belongs to $P_{t}^{2}$ for any $t>0,$ and
\begin{align*}
\limsup_{t\rightarrow +\infty}\Biggl[\sup_{x\in B}\Bigl\|{\mathbb F}_{1}(t,\cdot)\varphi(|P(\cdot;x)|)\Bigr\|_{P_{t}^{1}} +\sup_{x\in B}\Bigl\|{\mathbb F}_{2}(t,\cdot)\varphi \bigl(\|G(\cdot;x)\|_{Y}\bigr)\Bigr\|_{P_{t}^{2}}\Biggr]=0,
\end{align*}
for any trigonometric polynomial [$c$-periodic function; $(c_{j})_{j\in {\mathbb N}_{n}}$-periodic function] $P(\cdot;\cdot).$
\end{itemize}
If $F\in e-({\mathcal B},\phi,{\mathbb F}_{1})-B^{{\mathcal P}_{\cdot}^{1}}(\Lambda \times X : {\mathbb C})$ [$e-({\mathcal B},\phi,{\mathbb F}_{1})-B^{{\mathcal P}_{\cdot}^{1}}_{c_{1}{\rm I}}(\Lambda \times X : {\mathbb C});$ $e-({\mathcal B},\phi,{\mathbb F}_{1})_{j\in {\mathbb N}_{n}}-B^{{\mathcal P}_{\cdot}^{1}}_{c_{1}^{j}{\rm I}}(\Lambda \times X : {\mathbb C})$]
and
$G\in e-({\mathcal B},\phi,{\mathbb F}_{2})-B^{{\mathcal P}_{\cdot}^{2}}(\Lambda \times X : Y)$ [$e-({\mathcal B},\phi,{\mathbb F}_{2})-B^{{\mathcal P}_{\cdot}^{2}}_{c_{2}{\rm I}}(\Lambda \times X : Y);$ $e-({\mathcal B},\phi,{\mathbb F}_{2})_{j\in {\mathbb N}_{n}}-B^{{\mathcal P}_{\cdot}^{2}}_{c_{2}^{j}{\rm I}}(\Lambda \times X : Y)$], then $F\cdot G\in e-({\mathcal B},\phi,{\mathbb F})-B^{{\mathcal P}_{\cdot}}(\Lambda \times X : Y)$ [$e-({\mathcal B},\phi,{\mathbb F})-B^{{\mathcal P}_{\cdot}}_{c_{1}{\rm I}}(\Lambda \times X : Y);$ $e-({\mathcal B},\phi,{\mathbb F})_{j\in {\mathbb N}_{n}}-B^{{\mathcal P}_{\cdot}}_{c_{j}{\rm I}}(\Lambda \times X : Y)$] with $F=F_{1}F_{2};$ here $c_{1},\ c_{2}\in  {\mathbb C}\setminus \{0\}$ and $c_{1}^{j},\ c_{2}^{j}\in  {\mathbb C}\setminus \{0\}$ for all $j\in {\mathbb N}_{n}.$
\end{prop}

In connection with Proposition \ref{approxb}, let us notice that the existence and uniqueness of Besicovitch-$p$-almost periodic solutions for certain classes of PDEs on some proper subdomains of ${\mathbb R}^{n}$ have recently been analyzed in \cite{multi-besik}. The interested reader may try to provide some applications of Proposition \ref{approxb} in the analysis of the existence and uniqueness of metrical Besicovitch almost periodic solutions for certain classes of PDEs; see, e.g., the seventh application from \cite[Section 4]{multi-besik}, where the Besicovitch-$p$-almost periodicity of solutions is clarified for the equations depending on two variables, on the sectors of form $\Lambda=(-\infty,0] \times {\mathbb R}$ or $\Lambda=[0,\infty) \times {\mathbb R}.$ Let us also emphasize that the equation \eqref{helder-aps} can be viewed as the abstract H\"older inequality in pseudometric spaces.

In \cite[Section 3]{multi-besik}, we have revisited the important research studies \cite{doss}-\cite{doss1} by R. Doss. We close this section with the observation that we will not analyze here the metrical analogues of conditions (A), (AS), (A)$_{\infty}$ and (B) considered therein.

\section{Further results and applications}\label{maref}

Concerning the convolution invariance of function spaces introduced in this paper, we will only 
present here a few comments and a concrete application in the study of the existence and uniqueness of metrically Besicovitch almost periodic solutions of the heat equation in ${\mathbb R}^{n};$ the theoretical analysis is very similar to the analysis of the invariance under the actions of infinite convolution products carried out in the next subsection (see \cite{metrical,metrical-stepanov,metrical-weyl} for some results obtained recently in this direction). 

It is well known that 
a unique solution of the heat equation $u_{t}(t,x)=u_{xx}(t,x),$ $t\geq 0,\ x\in {\mathbb R}^{n};$ $u(0,x)=F(x),$ $x\in {\mathbb R}^{n}$ is given by the action of 
Gaussian semigroup 
\begin{align}\label{gauss}
F\mapsto (G(t)F)(x)\equiv \bigl(4\pi t\bigr)^{-n/2}\int_{{\mathbb R}^{n}}e^{-|y|^{2}/4t}F(x-y)\, dy,\quad t>0,\ x\in {\mathbb R}^{n}.
\end{align}
Suppose now that there exist two finite real numbers $b\geq 0$ and $c>0$ such that $|F(x)|\leq c(1+|x|)^{b},$ $x\in {\mathbb R}^{n}$ as well as that $a>0,$ $\alpha>0,$
$1\leq p<+\infty, $
$\alpha p\geq 1,$ and $1/(\alpha p)+1/q=1$. Suppose, further, that $\nu \in L^{p}_{loc}({\mathbb R}^{n})$, there exist finite real numbers
$M_{0}>0$ and $t_{0}>0$ such that 
\begin{align}\label{szszn}
\int_{[-t,t]^{n}}\nu^{p}(s)\, ds\leq M_{0}t^{ap},\quad t\geq t_{0},
\end{align} 
as well as there exists a function $\varphi: {\mathbb R}^{n} \rightarrow [0,\infty)$ such that $\nu(x)\leq \nu(x-y)\varphi(y)$ for all $x,\ y\in {\mathbb R}^{n}$ and
\begin{align*}
\int_{{\mathbb R}^{n}}e^{-c|y|^{2}}\bigl( 1+|y|^{ap} \bigr)\varphi^{p}(y)\, dy<+\infty,
\end{align*}
for any $c>0.$ Let for each $t>0$ we have $P_{t}:=L^{p}_{\nu}([-t,t]^{n}),$ and let $d_{t}$ be the metric induced by the norm of this Banach space.
Let us fix a real number $t_{0}$ in \eqref{gauss}, and let $F\in e-(x^{\alpha},t^{-a})-B^{{\mathcal P}_{\cdot}}({\mathbb R}^{n} : {\mathbb C}).$ Then the mapping $x\mapsto (G(t_{0})F)(x),$ $x\in {\mathbb R}^{n}$ is well-defined and has the same growth as $f(\cdot);$ let us prove that this mapping belongs to the class $e-(x^{\alpha},t^{-a})-B^{{\mathcal P}_{\cdot}}({\mathbb R}^{n} : {\mathbb C})$  as well. Let $\epsilon>0$ be given in advance. Set, as in \cite{multi-besik},  
$$
c_{t_{0}}:=\bigl(4\pi t_{0}\bigr)^{-n/2}\Bigl\| e^{-|\cdot|^{2}/8t_{0}}\Bigr\|_{L^{q}({\mathbb R}^{n})}^{\alpha p}.
$$
By our assumption, there exist a trigonometric polynomial $P(\cdot)$ and a finite real number $t_{1}>0$ such that
$$
\int_{[-t,t]^{n}}\bigl| F(x)-P(x) \bigr|^{\alpha p}\nu^{p}(x)\, dx<\epsilon_{0}t^{ap},\quad t\geq t_{1}.
$$
The function $x\mapsto (G(t_{0})P)(x),$ $x\in {\mathbb R}^{n}$ is Bohr almost periodic and the required conclusion simply follows from the next computation:% (see also the computation from the proof of Proposition \ref{stan}):
\begin{align}
\notag \frac{1}{t^{ap}}&\int_{[-t,t]^{n}}\bigl| (G(t_{0})F)(x)-(G(t_{0})P)(x) \bigr|^{\alpha p}\nu^{p}(x)\, dx
\\\label{dojaja}& \leq \frac{c_{t_{0}}}{t^{ap}}\int_{[-t,t]^{n}}\int_{{\mathbb R}^{n}} e^{-|y|^{2}\alpha p/8t_{0}}\bigl| F(x-y)-P(x-y) \bigr|^{\alpha p}\, dy \, \cdot \nu^{p}(x)\, dx
\\\notag & =\frac{c_{t_{0}}}{t^{ap}}\int_{{\mathbb R}^{n}}e^{-|y|^{2}\alpha p/8t_{0}}\int_{[-t,t]^{n}} \bigl| F(x-y)-P(x-y) \bigr|^{\alpha p}\nu^{p}(x)\, dx\, dy
\\\notag & \leq  \frac{c_{t_{0}}}{t^{ap}}\int_{{\mathbb R}^{n}}e^{-|y|^{2}\alpha p/8t_{0}}\int_{[-t+|y|,t+|y|]^{n}} \bigl| F(x)-P(x) \bigr|^{\alpha p}\nu^{p}(x)\, dx\, dy
\\\notag & \leq  \frac{c_{t_{0}}}{t^{ap}}\int_{{\mathbb R}^{n}}e^{-|y|^{2}\alpha p/8t_{0}}\Biggl[\int_{[-t+|y|,t+|y|]^{n}} \bigl| F(x)-P(x) \bigr|^{\alpha p}\nu^{p}(x-y)\, dx\Biggr]\varphi^{p}(y)\, dy
\\\notag & \leq \frac{c_{t_{0}}}{t^{ap}}\int_{{\mathbb R}^{n}}e^{-|y|^{2}\alpha p/8t_{0}}\epsilon_{0}2^{ap}\bigl(t^{ap}+|y|^{ap} \bigr)\varphi^{p}(y)\,\, dy,\quad t\geq t_{1}.
\end{align}

The argumentation employed for proving the estimate \eqref{dojaja} can be used to prove some results about the invariance of certain types of metrical Besicovitch almost periodicity under the actions of the usual convolution product 
\begin{align*}
f\mapsto F(x)\equiv \int_{{\mathbb R}^{n}}h(x-y)f(y)\, dy,\quad x\in {\mathbb R}^{n},
\end{align*}
provided that the function $h\in L^{1}({\mathbb R}^{n})$ has a certain growth order. For instance, an extension of \cite[Theorem 4.6]{multi-besik} can be proved in this context (the use of general binary relations $\rho$ is important, which can be seen from our recent applications to the Gaussain semigroup given in \cite{metrical}; we will not reconsider such applications here). 
Furthermore, using a similar idea as above, we can consider the existence and uniqueness of metrical Besicovitch almost periodic solutions for some special classes of evolution equations of first order; for instance, we can analyze the evolution systems in the space $Y:=L^{r}({\mathbb R}^{n}),$ where $r\in [1,\infty),$ generated by the family of operators $ A(t):= \Delta +a(t){\rm I}$, $t \geq 0$, where $\Delta$ is the Dirichlet Laplacian on $L^{r}(\mathbb{R}^{n})$  and $ a \in L^{\infty}([0,\infty)) .$ See the sixth application from \cite[Section 4]{multi-besik} for more details.

\subsection{Invariance under the actions of infinite convolution products}\label{seka}

In this subsection, we will present a few results concerning the invariance of introduced classes of functions under the actions of infinite convolution product
\begin{align}\label{trigpol}
t\mapsto F(t):=\int^{t}_{-\infty}R(t-s)f(s)\, ds,\quad t\in {\mathbb R};
\end{align}
for simplicity, we will not consider the multi-dimensional case here (cf. \cite{nova-selected} for more details). 

Our first result considers $(c,\Lambda')$-almost periodic functions in variation; we will use the following notion: Suppose that $\emptyset \neq \Lambda' \subseteq {\mathbb R}$ and $c\in {\mathbb C},$ $|c|=1.$ A continuous function $f: {\mathbb R} \rightarrow Y$ is called $(c,\Lambda')$-almost periodic in variation if and only if for each $\epsilon>0$ there exists a finite real number $l>0$ such that, for each $t_{0}\in \Lambda'$, there exists a number $\tau \in [t_{0}-l,t_{0}+l]$ such that
\begin{align}\label{kasnim}
\sup_{t\in {\mathbb R}} \Bigl( |f(t+\tau)-cf(t)|+V_{1}(f(\cdot +\tau)-cf(\cdot);t) \Bigr)<\epsilon.
\end{align}

We are ready to state the following simple result:

\begin{prop}\label{stojke}
Suppose that $(R(t))_{t>0}\subseteq L(X,Y)$ is a strongly continuous operator family and $\int^{\infty}_{0}\|R(t)\|\, dt<+\infty. $ If the function $f: {\mathbb R} \rightarrow X$ is $(c,\Lambda')$-almost periodic in variation ($\emptyset \neq \Lambda' \subseteq {\mathbb R};$ $c\in {\mathbb C},$ $|c|=1$), then the function $F(\cdot),$ given by \eqref{trigpol}, is likewise $(c,\Lambda')$-almost periodic in variation.
\end{prop}

\begin{proof}
Let $\epsilon>0$ be given. Then there exists a finite real number $l>0$ such that, for each $t_{0}\in \Lambda'$, there exists a number $\tau \in [t_{0}-l,t_{0}+l]$ such that \eqref{kasnim} holds (cf. also \cite[Theorem 6.1.53]{nova-selected}). We need to show that \eqref{kasnim} holds with the function $f(\cdot)$ replaced by the function $F(\cdot)$ therein. But, keeping in mind the elementary definitions, it is very elementary to prove that
\begin{align*}
\sup_{t\in {\mathbb R}}& \Bigl( \|f(t+\tau)-cf(t)\|+V_{1}(f(\cdot +\tau)-cf(\cdot);t) \Bigr) 
\\& \leq \sup_{t\in {\mathbb R}} \Bigl( \|F(t+\tau)-cF(t)\|_{Y}+V_{1}(F(\cdot +\tau)-cF(\cdot);t) \Bigr) \cdot \int^{\infty}_{0}\|R(s)\|\, ds.
\end{align*}
This implies the required result.
\end{proof}

We continue by stating the following analogue of \cite[Theorem 3.13]{metrical}:

\begin{prop}\label{szszc}
Suppose that $(R(s))_{s> 0}\subseteq L(X,Y)$ is a strongly continuous operator family, $c\in {\mathbb C}\setminus \{0\}$ and $\emptyset \neq \Lambda'\subseteq {\mathbb R}.$
Let $P:=C_{b,\nu}({\mathbb R} : [0,\infty))$ and $d(f,g):=\| f-g\|_{C_{b,\nu}({\mathbb R} :[0,\infty))}$ for all $f,\ g\in P,$ and let there exist a function $\varphi : {\mathbb R} \rightarrow [0,\infty)$ continuous at the point zero and satisfying that $\phi(xy)\leq \phi(x)\varphi(y)$
for all $x,\ y\in {\mathbb R}.$ 
Suppose, further, that
$\int_{(0,\infty)}\|R(s )\|\, ds<\infty $ and $\int_{(0,\infty)}\frac{\|R(s )\|}{\nu(x-s)}\, ds<\infty$
for all $x\in {\mathbb R}.$
If $f : {\mathbb R} \rightarrow X$ is a bounded, continuous and Bohr $(\phi,{\mathbb F},\Lambda',c,{\mathcal P})$-almost periodic function, then the function $F: {\mathbb R} \rightarrow Y,$ given by \eqref{trigpol}, is
bounded, continuous and Bohr $(\phi,{\mathbb F}_{1},\Lambda',c,{\mathcal P}_{1})$-almost periodic, provided that $P_{1}$ is a Banach space,
\begin{align}\label{lozolozo}
{\mathbb F}_{1}(\cdot)\phi\Biggl(  \int^{\infty}_{0} \| R(s)\|  \frac{ds}{\nu(\cdot-s)}\Biggr)\in P_{1}
\end{align}
and condition \emph{(C0)} holds for $P_{1}.$
\end{prop}

\begin{proof}
Since $\int_{0}^{\infty}\|R(s )\|\, ds<\infty $, a very simple argumentation from our previous research studies shows that the function $F(\cdot)$ is bounded and continuous. Let us show that $F(\cdot)$ is Bohr $(\phi,{\mathbb F}_{1},\Lambda',c,{\mathcal P}_{1})$-almost periodic. Let $\epsilon>0$ be arbitrary and let $\tau\in \Lambda'$ be a corresponding $\epsilon$-period of the function $f(\cdot).$ Then the final conclusion simply follows from the next computation involving the conditions $\int_{0}^{\infty}\frac{\|R(s )\|}{\nu(x-s)}\, ds<\infty$
for all $x\in {\mathbb R}$, \eqref{lozolozo} and the continuity of function $\varphi(\cdot)$ at the point zero: 
\begin{align*}
\Biggl\|& {\mathbb F}_{1}(\cdot)\phi\Biggl( \Bigl\| \int^{\infty}_{0}R(s)[f(\cdot +\tau-s)-cf(\cdot-s)]\, ds\Bigr\|_{Y} \Biggr)\Biggr\|_{P_{1}}
\\\leq & \Biggl\| {\mathbb F}_{1}(\cdot)\phi\Biggl(  \int^{\infty}_{0} \| R(s)\| \cdot \|f(\cdot +\tau-s)-cf(\cdot-s)\|\, ds \Biggr)\Biggr\|_{P_{1}}
\\\leq &\Biggl\| {\mathbb F}_{1}(\cdot)\phi\Biggl(  \int^{\infty}_{0} \| R(s)\| \cdot \frac{\epsilon}{\nu(\cdot-s)}\, ds\Biggr)\Biggr\|_{P_{1}}
\\\leq & \Biggl\| {\mathbb F}_{1}(\cdot)\varphi(\epsilon)\phi\Biggl(  \int^{\infty}_{0} \| R(s)\| \cdot \frac{ds}{\nu(\cdot-s)} \Biggr)\Biggr\|_{P_{1}}.
\end{align*}
\end{proof}

Now we will assume that
the operator family $(R(t))_{t>0}\subseteq L(X,Y)$ satisfies that there exist finite real constants $M>0$, $\beta \in (0,1]$ and $\gamma >1$ such that
\begin{align}\label{rad}
\bigl\| R(t)\bigr\|_{L(X,Y)}\leq M\frac{t^{\beta-1}}{1+t^{\gamma}},\quad t>0.
\end{align}

The following result is a slight extension of \cite[Proposition 4.1]{multi-besik}:

\begin{prop}\label{stanb}
Suppose that the operator family $(R(t))_{t>0}\subseteq L(X,Y)$ satisfies \eqref{rad}, as well as that  $a>0,$ $\alpha>0,$ $1\leq p<+\infty,$ $\alpha p \geq 1,$ $ap\geq 1,$
$\alpha p(\beta-1)/(\alpha p-1)>-1$ if $\alpha p>1$, and $\beta=1$ if $\alpha p=1.$ Suppose, further, that the function $f : {\mathbb R} \rightarrow X$ is Stepanov-$(\alpha p)$-bounded, i.e.,
$$
\bigl\|f\bigr\|_{S^{p}}:=\sup_{t\in {\mathbb R}}\int^{t+1}_{t}\bigl\|f(s)\bigr\|^{\alpha p}\, ds<+\infty,
$$
as well as that the function $\nu : {\mathbb R} \rightarrow (0,\infty)$ is monotonically decreasing, Stepanov-$p$-bounded and satisfies that the function $f(\cdot)\nu^{1/\alpha}(\cdot)$ is Stepanov-$(\alpha p)$-bounded as well as that there exist finite real numbers
$M_{0}>0$ and $t_{0}>0$ such that 
\eqref{szszn} holds with $n=1.$
Let for each $t>0$ we have $P_{t}:=L^{p}_{\nu}([-t,t]),$ and let $d_{t}$ be the metric induced by the norm of this Banach space.
If $f\in e-(x^{\alpha},t^{-a})-B^{{\mathcal P}_{\cdot}}({\mathbb R} : X),$
then the function $F(\cdot)$, given by \eqref{trigpol}, is bounded, continuous and belongs to the class $ e-(x^{\alpha},t^{-a})-B^{{\mathcal P}_{\cdot}}({\mathbb R} : Y).$
\end{prop}

\begin{proof}
The proof of \cite[Proposition 2.6.11]{nova-mono} indicates that the function $F(\cdot)$ is well-defined, bounded and continuous. Let $(P_{k})$ be a sequence of trigonometric polynomials such that
\begin{align*}
\lim_{k\rightarrow +\infty}\limsup_{t\rightarrow +\infty}\frac{1}{2t^{ap}}\int^{t}_{-t}\bigl\| f(s)-P_{k}(s)\bigr\|^{\alpha p}\nu^{p}(s)\, ds=0.
\end{align*}
The function $t\mapsto F_{k}(t)\equiv \int^{t}_{-\infty}R(t-s)P_{k}(s)\, ds,$ $t\in {\mathbb R}$ is almost periodic due to the above-mentioned proposition. Since there exist finite real numbers
$M>0$ and $t_{0}>0$ such that \eqref{szszn} is true with $n=1$, the conclusions established in (BE2) hold. Hence,
we need to prove that
\begin{align}\label{dejanb}
\lim_{k\rightarrow +\infty}\limsup_{t\rightarrow +\infty}\frac{1}{2t^{ap}}\int^{t}_{-t}\bigl\| F(s)-F_{k}(s)\bigr\|^{\alpha p}\nu^{p}(s)\, ds=0.
\end{align}
In the remainder of the proof, we will only consider case $\alpha p>1.$ Suppose that $\zeta \in (1/(\alpha p),(1/(\alpha p))+\gamma-\beta).$ Then the function
$s\mapsto |s|^{\beta-1}(1+|s|)^{\zeta}/(1+|s|^{\gamma}),$
$s\in {\mathbb R}$ belongs to the space $L^{\alpha p/(\alpha p-1)}((-\infty,0));$ further on, since we have assumed that 
the function $\nu(\cdot)$ is Stepanov-$p$-bounded and the functions $f(\cdot),$ $f(\cdot)\nu^{1/\alpha}(\cdot)$ are Stepanov-$(\alpha p)$-bounded, 
the argumentation contained in the proof of \cite[Theorem 2.11.4]{nova-mono} shows that the function $s\mapsto (1+|s|)^{-\zeta}\| P_{k}(s+z)-f(s+z)\|\nu^{1/\alpha}(s+z),$
$s\in {\mathbb R}$ belongs to the space $L^{\alpha p}((-\infty,0))$ for all $k\in {\mathbb N}$ and $z\in {\mathbb R}.$ We have ($M_{1}>0$ is a finite real constant, $t>0$):{\small
\begin{align*}
&\frac{1}{2t^{ap}}\int^{t}_{-t}\bigl\| F(s)-F_{k}(s)\bigr\|^{\alpha p}\nu^{p}(s)\, ds  
\\& \leq \frac{1}{2t^{ap}}\int^{t}_{-t}\Biggl|\int^{0}_{-\infty}\|R(-z)\| \cdot \bigl\| P_{k}(s+z)-f(s+z)\bigr\| \, dz \Biggr|^{\alpha p}\nu^{p}(s)\, ds
\\& \leq \frac{M}{2t^{ap}}\int^{t}_{-t}\Biggl|\int^{0}_{-\infty}
\frac{|z|^{\beta-1}(1+|z|)^{\zeta}}{(1+|z|^{\gamma})}\cdot (1+|z|)^{ -\zeta}\bigl\| P_{k}(s+z)-f(s+z)\bigr\| \nu^{1/\alpha}(s+z)\, dz \Biggr|^{\alpha p}\, ds
\\& \leq \frac{M_{1}}{2t^{ap}}\int^{t}_{-t}\int^{0}_{-\infty}\frac{1}{(1+|z|^{\alpha \zeta})^{p}}\bigl\| P_{k}(s+z)-f(s+z)\bigr\|^{\alpha p}\nu^{p}(s+z)\, dz \, ds.
\end{align*}}
The estimate \eqref{dejanb} then follows from the remainder of the long computation carried out in the proof of \cite[Proposition 4.1]{multi-besik}.
\end{proof}

In the following slight extension of \cite[Proposition 4.2]{multi-besik}, the inhomogeneity $f(\cdot)$ is not necessarily Stepanov-$(\alpha p)$-bounded and the weight $\nu(\cdot)$ is not necessarily Stepanov-$p$-bounded. The proof is almost the same as the proof of Proposition \ref{stanb} and the above-mentioned results from \cite{multi-besik}:

\begin{prop}\label{stan1b}
Suppose that the operator family $(R(t))_{t>0}\subseteq L(X,Y)$ satisfies \eqref{rad}, as well as that $a>0,$ $\alpha>0,$ $1\leq p<+\infty,$ $\alpha p \geq 1,$ $ap\geq 1,$
$\alpha p(\beta-1)/(\alpha p-1)>-1$ if $\alpha p>1$, and $\beta=1$ if $\alpha p=1.$ Suppose, further, that
there exists a finite real constant $M>0$ such that $\|f(t)\|\leq M(1+|t|)^{b},$ $t\in {\mathbb R}$ for some real constant $b \in [0,\gamma-\beta)$ and a Lebesgue measurable function $f(\cdot),$
the function $\nu : {\mathbb R} \rightarrow (0,\infty)$ is monotonically decreasing, 
there exists a finite real constant $M'>0$ such that $\|\nu(t)\|\leq M'(1+|t|)^{b/\alpha},$ $t\in {\mathbb R}$ and there exist finite real numbers
$M_{0}>0$ and $t_{0}>0$ such that \eqref{szszn} holds with $n=1.$ 
Let for each $t>0$ we have $P_{t}:=L^{p}_{\nu}([-t,t]),$ and let $d_{t}$ be the metric induced by the norm of this Banach space.
If $f\in e-(x^{\alpha},t^{-a})-B^{{\mathcal P}_{\cdot}}({\mathbb R} : X),$ then the function $F(\cdot)$, given by \eqref{trigpol}, is continuous, belongs to the class $e-(x^{\alpha},t^{-a})-B^{{\mathcal P}_{\cdot}}({\mathbb R} : Y),$ and there exists a finite real constant $M'>0$ such that $\|F(t)\|_{Y}\leq M'(1+|t|)^{b},$ $t\in {\mathbb R}.$
\end{prop}

As mentioned in a great number of our recent research articles, Proposition \ref{stojke}, Proposition \ref{stanb} and Proposition \ref{stan1b} can be applied to a large class of the abstract (degenerate) Volterra integro-differential equations without initial conditions. Here we will only note that we can apply these results in the analysis of the existence and uniqueness of metrical Besicovitch-$p$-almost periodic type solutions of 
the initial value problems with constant coefficients
\[
\begin{array}{l}
D_{t,+}^{\gamma}u(t,x)=\sum_{|\alpha|\leq k}a_{\alpha}f^{(\alpha)}(t,x)+f(t,x),\ t\in {\mathbb R},\ x\in \mathbb{R}^n
\end{array}
\] 
in the space $L^{p}(\mathbb{R}^n),$ where 
$\gamma \in (0,1),$
$D_{t,+}^{\gamma}u(t)$ denotes the Weyl-Liouville fractional derivative of order $\gamma$ and
$1\leq p<\infty .$
See also \cite{nova-mono} for many other applications of this type. 

The statement of \cite[Theorem 4.5]{multi-besik}, which concerns the existence and uniqueness of Besicovitch-$p$-almost periodic solutions of the abstract nonautonomous
differential equations of first order, can be simply reformulated in our new context, with the use of the same pivot Banach spaces $P_{t}=L^{p}_{\nu}([-t,t])$ for all $t>0.$

For simplicity, we will not consider here the invariance of various classes of metrical Stepanov almost periodic type functions and metrical Weyl almost periodic type functions under the actions of infinite convolution products; see \cite[Subsection 2.1]{metrical-stepanov} and  \cite[Subsection 3.1]{metrical-weyl} for some results obtained in this direction. Because of a certain similarity with our previous research studies, we will not reconsider here the classical solutions of the inhomogeneous wave equation given by the famous d'Alembert formula (the Poisson formula; the  Kirchhoff formula), as well; see \cite{nova-selected} for more details about the subject.

\section{Conclusions and final remarks}\label{toi}

In this paper, we have investigated metrical approximations by trigonometric polynomials and $\rho$-periodic type functions, providing also certain applications to
the abstract Volterra integro-differential equations and the partial differential equations. In this section, we will provide several comments and final remarks about the introduced notion and the obtained results.

First of all, we would like to emphasize that, in our definitions given in Section \ref{docolord}, we have taken the norm in $Y$ of certain terms and assumed that  $P \subseteq [0,\infty)^{\Lambda}.$ Without going into full details, we will only note here that we can also consider the following notion: Assume that $\emptyset \neq \Lambda \subseteq {\mathbb R}^{n},$ 
$\phi_{Y} : Y \rightarrow Y,$ ${\mathbb F} : \Lambda \rightarrow (0,\infty)$,
$P_{Y} \subseteq Y^{\Lambda},$ the zero function belongs to $P_{Y}$, and 
${\mathcal P}_{Y}=(P_{Y},d_{Y})$ is a pseudometric space.
\begin{itemize}\index{function!strongly $(\phi_{Y},{\mathbb F},{\mathcal B},{\mathcal P}_{Y})$-almost periodic}\index{function!semi-$(\phi_{Y},\rho,{\mathbb F},{\mathcal B},{\mathcal P}_{Y})$-periodic}\index{function!semi-$(\phi_{Y},\rho_{j},{\mathbb F},{\mathcal B},{\mathcal P}_{Y})_{j\in {\mathbb N}_{n}}$-periodic}
\item[(i)] We say that the function $F(\cdot;\cdot)$ is strongly $(\phi_{Y},{\mathbb F},{\mathcal B},{\mathcal P}_{Y})$-almost periodic (semi-$(\phi_{Y},\rho,{\mathbb F},{\mathcal B},{\mathcal P}_{Y})$-periodic, semi-$(\phi_{Y},\rho_{j},{\mathbb F},{\mathcal B},{\mathcal P}_{Y})_{j\in {\mathbb N}_{n}}$-periodic) if and only if for each $B\in {\mathcal B}$ there exists a sequence $(P_{k}^{B}({\bf t};x))$ of trigonometric polynomials ($\rho$-periodic functions, $(\rho_{j})_{j\in {\mathbb N}_{n}}$-periodic functions)
such that 
$$
\lim_{k\rightarrow +\infty}\sup_{x\in B} \Bigl\|{\mathbb F}(\cdot) \phi\Bigl(P_{k}^{B}(\cdot;x)-F(\cdot;x)\Bigr)\Bigr\|_{P_{Y}}=0.
$$ \index{function!Bohr $(\phi_{Y},{\mathbb F},{\mathcal B},\Lambda',\rho,{\mathcal P}_{Y})$-almost periodic}
\item[(ii)] We say that the function $F(\cdot;\cdot)$ is Bohr $(\phi_{Y},{\mathbb F},{\mathcal B},\Lambda',\rho,{\mathcal P}_{Y})$-almost periodic if and only if for every $B\in {\mathcal B}$ and $\epsilon>0$
there exists $l>0$ such that for each ${\bf t}_{0} \in \Lambda'$ there exists ${\bf \tau} \in B({\bf t}_{0},l) \cap \Lambda'$ such that, for every ${\bf t}\in \Lambda$ and $x\in B,$ there exists an element $y_{{\bf t};x}\in \rho (F({\bf t};x))$ such that
\begin{align*}
\sup_{x\in B} \Biggl\| {\mathbb F}(\cdot) \phi\Bigl(F(\cdot+{\bf \tau};x)-y_{\cdot;x}\Bigr)\Biggr\|_{P_{Y}} \leq \epsilon .
\end{align*}\index{function!$(\phi_{Y},{\mathbb F},{\mathcal B},\Lambda',\rho,{\mathcal P}_{Y})$-uniformly recurrent}
\item[(iii)] We say that the function 
$F(\cdot;\cdot)$ is $(\phi_{Y},{\mathbb F},{\mathcal B},\Lambda',\rho,{\mathcal P}_{Y})$-uniformly recurrent if and only if for every $B\in {\mathcal B}$ 
there exists a sequence $({\bf \tau}_{k})$ in $\Lambda'$ such that $\lim_{k\rightarrow +\infty} |{\bf \tau}_{k}|=+\infty$ and that, for every ${\bf t}\in \Lambda$ and $x\in B,$ there exists an element $y_{{\bf t};x}\in \rho (F({\bf t};x))$ such that 
\begin{align*}
\lim_{k\rightarrow +\infty}\sup_{x\in B} \Biggl\| {\mathbb F}(\cdot)\phi\Bigl( F(\cdot+{\bf \tau}_{k};x)-y_{\cdot;x}\Bigr)\Biggr\|_{P_{Y}}=0.
\end{align*}
\end{itemize}

It seems very plausible that many structural results from Section \ref{docolord} can be reformulated for these classes of functions. Details can be left to the interested readers.

It is our strong belief that this paper is only a beginning of serious investigations of metrical approximations of functions and their applications. At the end of paper, we will only mention a few more topics not considered here:

1. In this paper, we have not considered extensions of functions obtained as metrical approximations by trigonometric polynomials and $\rho$-periodic type functions.

2. There are several different ways to introduce the notion of a function of bounded
variation in 
multiple dimensions; for more details about this important subject, we refer the reader to the master thesis \cite{brener} by S. Breneis and references cited therein.
We will analyze multi-dimensional (metrical) almost periodic type functions in variation and multi-dimensional (metrical) H\"older almost periodic type functions somewhere else.

3. In our analysis, we do not require that the set $\Lambda$ is unbounded. If the set $\Lambda$ is bounded, the situation is completely without control and a series of further investigations can be carried out; in both situations, we need to further explore several new classes of functions obtained by specifying the considered pseudometric spaces ${\mathcal P}$.

4. Let us recall that 
the composition principles for the metrical Stepanov $c$-almost periodic type functions and the metrical Weyl $c$-almost periodic type functions have been considered in \cite[Theorem 2.6]{metrical-stepanov} and \cite[Theorem 3.7]{metrical-weyl}, respectively; a composition principle for Besicovitch-$p$-almost
periodic type functions has been deduced in \cite[Theorem 2.10]{multi-besik} following the approach of M. Ayachi and J. Blot \cite[Lemma 4.1]{ayachi}. It is also worth noting that
the notion of a $(c,\Lambda')$-uniformly recurrent function in $p$-variation can be also introduced ($1\leq p<+\infty$); then an analogue of the composition principle stated in \cite[Theorem 2.28]{sds} can be proved under certain very restrictive assumptions (the interested reader may try to reconsider the results established in our recent research study \cite{sds} by M. T. Khalladi et al. for $c$-almost periodic type functions in $p$-variation and H\"older $c$-almost periodic type functions, which can defined in a similar fashion). We will not analyze the composition principles for the introduced classes of functions and related applications to the semilinear Cauchy problems here.


\begin{thebibliography}{99}

\bibitem{deda}
J. Andres, A. M. Bersani, R. F. Grande,
Hierarchy of almost-periodic function spaces,
Rend. Mat. Appl. (7) {\textbf 26} (2006), 121-188.

\bibitem{andres}
J. Andres, D. Pennequin,
Semi-periodic solutions of difference and differential equations,
Boundary Value Problems {\bf 141} (2012), 1--16.

\bibitem{ayachi}
M. Ayachi, J. Blot, Variational methods for almost periodic solutions of a class
of neutral delay equations, Abstract Appl. Anal. 2008, Article ID 153285, 13 pages
doi:10.1155/2008/153285.

\bibitem{besik}
A. S. Besicovitch,
Almost Periodic Functions,
Dover Publ., New York, 1954.

\bibitem{brener}
S. Breneis,
Functions of bounded
variation in one and
multiple dimensions,
Master Thesis, Johannes Kepler Universit\"at Linz, 2020.

\bibitem{chaouchi}
B. Chaouchi, M. Kosti\' c, S. Pilipovi\'c, D. Velinov,
Semi-Bloch periodic functions, semi-anti-periodic functions and applications,
Chelj. Phy. Math. J. {\bf 5} 92020), 243--255.

\bibitem{marko-manuel-ap}
A. Ch\'avez, K. Khalil, M. Kosti\'c, M. Pinto,
Multi-dimensional almost periodic type functions and applications, 
submitted. 2020. arXiv:2012.00543.

\bibitem{diagana}
T. Diagana,
Almost Automorphic Type and Almost Periodic Type Functions in Abstract Spaces,
Springer-Verlag, New York, 2013.

\bibitem{variable}
L. Diening, P. Harjulehto, P. H\"ast\"uso, M. Ruzicka,
Lebesgue and Sobolev Spaces with Variable
Exponents,
Lecture Notes in Mathematics, Springer, Heidelberg, 2011.

\bibitem{doss}
R. Doss,
On generalized almost periodic functions,
Annals of Math. {\bf 59}
(1954), 477--489.

\bibitem{doss1}
R. Doss,
On generalized almost periodic functions-II,
J. London Math. Soc.
{\bf 37} (1962), 133--140.

\bibitem{rho}
M. Fe\v{c}kan, M. T. Khalladi, M. Kosti\' c, A. Rahmani,
Multi-dimensional $\rho$-almost periodic type functions and applications,
submitted. 2021. arXiv:2109.10223.

\bibitem{fink}
A. M. Fink,
Almost Periodic Differential Equations,
Springer-Verlag, Berlin, 1974.

\bibitem{gaston}
G. M. N'Gu\' er\' ekata,
Almost Automorphic and Almost Periodic Functions
in Abstract Spaces,
Kluwer Acad. Publ, Dordrecht, 2001.

\bibitem{haraux}
A. Haraux, P. Souplet,
An example of uniformly recurrent function which is not almost periodic,
J. Fourier Anal. Appl. {\bf 10} (2004), 217--220.

\bibitem{sds}
M. T. Khalladi, M. Kosti\' c, M. Pinto, A. Rahmani, D. Velinov, 
$c$-Almost periodic type functions and applications,
Nonauton. Dyn. Syst. {\bf 7} (2020), 176--193.

\bibitem{nova-mono}
M. Kosti\'c,
Almost Periodic and Almost Automorphic Type Solutions to Integro-Differential Equations,
W. de Gruyter, Berlin, 2019.

\bibitem{nova-selected}
M. Kosti\'c,
Selected Topics in Almost Periodicity,
W. de Gruyter, Berlin, 2022.

\bibitem{brno}
M. Kosti\'c,
Generalized $c$-almost periodic type functions in ${\mathbb R}^{n}$,
Arch. Math. (Brno) {\bf 57} (2021), 221--253.

\bibitem{ejmaa-2022}
M. Kosti\'c,
Stepanov and Weyl classes of multi-dimensional $\rho$-almost periodic type functions,
Electronic J. Math. Anal. Appl. {\bf 10} (2022), 11--35.

\bibitem{metrical}
M. Kosti\'c,
Metrical almost periodicity and applications,
submitted. arXiv:2111.14614.

\bibitem{metrical-stepanov}
M. Kosti\'c,
Stepanov $\rho$-almost periodic functions in general metric,
submitted (2021). https://www.researchgate.net/publication/357152220.

\bibitem{metrical-weyl}
M. Kosti\'c,
Weyl $\rho$-almost periodic functions in general metric,
submitted (2021). https://www.researchgate.net/publication/357129349.

\bibitem{multi-besik}
M. Kosti\'c,
Multi-dimensional Besicovitch almost periodic type functions and applications,
submitted. 2022.  arXiv:2202.10521.

\bibitem{doss-rn}
M. Kosti\'c, W.-S. Du, V. E. Fedorov,
Doss $\rho$-almost periodic type functions in ${\mathbb R}^{n},$
Mathematics {\bf 9} (2021), 2825. https://doi.org/10.3390/math9212825.

\bibitem{188}
M. Levitan,
Almost Periodic Functions,
G.I.T.T.L., Moscow, 1953 (in Russian).

\bibitem{levitan}
B. M. Levitan, V. V. Zhikov, 
Almost Periodic Functions and Differential Equations, 
Univ. Publ. House, Moscow, 1978, English
translation by Cambridge University Press, 1982.

\bibitem{nawrocki}
A. Nawrocki,
Diophantine approximations and almost
periodic functions,
Demonstr. Math. {\bf 50} (2017), 100--104.

\bibitem{oliaro}
A. Oliaro, L. Rodino, P. Wahlberg, 
Almost periodic pseudodifferential operators and Gevrey
classes, 
Ann. Mat. Pura Appl. {\bf 191} (2012), 725--760.

\bibitem{pankov}
A. A. Pankov,
Bounded and Almost Periodic Solutions of Nonlinear Operator
Differential Equations, Kluwer Acad. Publ., Dordrecht, 1990.

\bibitem{shubin32}
M. A. Shubin, 
Differential and pseudodifferential operators in spaces of
almost periodic functions, 
Math. USSR-Sb. {\bf 24} (1974), 547--573.

\bibitem{stoinski2}
S. Sto\' inski,
A connection between $H$-almost periodic functions and almost periodic functions
of other types, 
Funct. Approx. Comment. Math. {\bf 3} (1976), 205--223.

\bibitem{stoja}
S. Sto\' inski,
Real-valued functions almost periodic in variation,
Func. Approx.
{\bf 22} (1993), 141--148.

\bibitem{stoja1}
S. Sto\' inski,
$L_{\alpha}$-almost periodic functions,
Demonstratio Math.
{\bf 28} (1995), 689--696.

\bibitem{stoj-nov}
S. Sto\' inski,
 A note on $H$-almost periodic functions and $S^{P}$-almost periodic functions,
Demonstratio Math.
{\bf 29} (1996), 557--564.

\bibitem{fasc}
S. Sto\' inski,
Some remarks on spaces of almost periodic functions,
Fasc. Math. No. {\bf 31} (2001), 105--1115.

\bibitem{30}
S. Zaidman,
Almost-Periodic Functions in Abstract Spaces,
Pitman Research Notes in
Math, Vol. \textbf{126}, Pitman, Boston, 1985.

\end{thebibliography}
\end{document}